\documentclass[letterpaper, 10 pt, conference]{ieeeconf}  % Comment this line out if you need a4paper

\IEEEoverridecommandlockouts                              % This command is only needed if 
\usepackage{amsmath,amssymb,bm,bbm,mathrsfs,amscd}
\usepackage{amsfonts}
\usepackage{color}
\usepackage{dsfont}
\usepackage{graphicx}
\usepackage{epsfig}
\usepackage{subfigure}
\usepackage{algorithm}
\usepackage[noend]{algpseudocode}
\usepackage{tikz}
\usepackage{mathtools}

\newtheorem{theorem}{Theorem}

\newtheorem{proposition}{Proposition}
\newtheorem{lemma}{Lemma}

\newtheorem{remark}{Remark}

\newtheorem{standassumption}{Standing Assumption}
\usepackage[font=small,labelfont=bf]{caption}
\usepackage{subfigure}
\usepackage{graphicx}
\graphicspath{{Figures/}} 
\usepackage{acronym}
\usepackage{latexsym}
\usepackage{paralist}
\usepackage{wasysym}
\usepackage{xspace}
\usepackage{framed}
\usepackage{authblk}
\newcommand{\ssc}[1]{{\color{black}#1}}
\newcommand{\us}[1]{{\color{black}#1}}
\newcommand{\eps}{\varepsilon}
\newcommand{\from}{\colon}

\DeclareMathOperator{\diag}{diag}

\DeclareMathOperator{\dom}{dom}

\DeclareMathOperator{\Id}{Id}

\newcommand{\ce}{\mathtt{e}}

\newcommand{\bL}{\mathbf{L}}
\newcommand{\bI}{\mathbf{I}}

\newcommand{\bD}{\mathbf{D}}

\newcommand{\bW}{\mathbf{W}}

\newcommand{\xbold}{{\bf x}}

\newcommand{\ubold}{{\bf u}}

\newcommand{\mubold}{{\bm{\mu}}}
\newcommand{\lambdabold}{{\bm{\lambda}}}
\newcommand{\gammabold}{{\bm{\gamma}}}
\newcommand{\sigmabold}{{\bm{\sigma}}}
\newcommand{\taubold}{{\bm{\tau}}}

\renewcommand{\iff}{\Leftrightarrow}

\renewcommand{\emptyset}{\varnothing}

\newcommand{\scrA}{\mathcal{A}}
\newcommand{\scrB}{\mathcal{B}}

\newcommand{\scrF}{\mathcal{F}}

\newcommand{\scrG}{\mathcal{G}}

\newcommand{\scrM}{\mathcal{M}}
\newcommand{\scrN}{\mathcal{N}}

\newcommand{\setU}{\mathsf{U}}

\newcommand{\setX}{\mathsf{X}}

\renewcommand{\Pr}{\mathbb{P}}
\newcommand{\Ex}{\mathbb{E}}

\newcommand{\F}{{\mathbb{F}}}
\newcommand{\0}{\mathbf{0}}
\newcommand{\1}{\mathbf{1}}

\newcommand{\R}{\mathbb{R}}

\newcommand{\N}{\mathbb{N}}
\DeclareMathOperator{\zer}{zer}

\newcommand{\Zer}{\mathsf{Zer}}
\DeclareMathOperator{\prox}{prox}

\DeclarePairedDelimiter{\abs}{\lvert}{\rvert}

\DeclarePairedDelimiter{\inner}{\langle}{\rangle}
\DeclarePairedDelimiter{\norm}{\lVert}{\rVert}

\newacro{UBV}[UBV]{uniformly bounded variance}
\newacro{SO}[SO]{Stochastic Oracle}
\newacroplural{NE}[NE]{Nash equilibria}
\newacro{VI}{variational inequality}
\newacroplural{VI}[VIs]{variational inequalities}
\newacro{SVI}{stochastic variational inequality}
\newacroplural{SVI}[SVIs]{stochastic variational inequalities}
\newacro{i.i.d.}[i.i.d.]{independent and identically distributed}
\newacro{FOM}{First-order method}
\newacro{SFBF}{stochastic forward-backward-forward}
\newacro{FBF}{forward-backward-forward}
\newacro{SEG}{stochastic extragradient}
\newacro{SFB}{stochastic forward-backward}
\newacro{GNEP}{generalized Nash equilibrium problem}
\newacro{GNE}{generalized Nash equilibrium}
\newacro{RISFBF}{relaxed inertial stochastic forward-backward-forward}

\newcommand{\RR}{{\mathbb{R}}}
\newcommand{\NN}{{\mathbb{N}}}
\newcommand{\EE}{{\mathbb{E}}}

\newcommand{\JJ}{{\mathbb{J}}}

\newcommand{\mc}{\mathcal}

\newcommand{\op}{\operatorname}

\newcommand{\bs}{\boldsymbol}
\newcommand{\fineass}{\hfill\small$\square$}

\title{\LARGE \bf
A relaxed-inertial forward-backward-forward algorithm for Stochastic Generalized Nash equilibrium seeking
}
\date{\today}

\author{Shisheng Cui$^{1}$, Barbara Franci$^{2}$, Sergio Grammatico$^{2}$, Uday V. Shanbhag$^{1}$ and Mathias Staudigl$^{3}$% <-this % stops a space
\thanks{*This work was partially supported by NWO under research projects OMEGA (613.001.702) and P2P-TALES (647.003.003), and by the ERC under research project COSMOS (802348). Mathias Staudigl's research benefited from the support of the FMJH Program PGMO and from the support of EDF.}
\thanks{$^{1}$ Shisheng Cui and Uday V. Shanbhag are with the Department of Industrial and Manufacturing Engineering, Pennsylvania State University, University Park, PA 16802, USA 
{\tt\small udaybag;suc256@psu.edu} }

\thanks{$^{2}$ Barbara Franci and Sergio Grammatico are with the Delft Center for System and Control, TU Delft, The Netherlands,
        {\tt\small b.franci-1;s.grammatico@tudelft.nl}}%
\thanks{$^{3}$ Mathias Staudigl is with the Department of Data Science and Knowledge Engineering, Maastricht University, The Netherlands,
{\tt\small m.staudigl@maastrichtuniversity.nl}}
}

\begin{document}
\maketitle
\thispagestyle{empty}
\pagestyle{empty}

%---------------------------------------------
%%% ABSTRACT
%----------------------------------------------------------------------
\begin{abstract}
In this paper we propose a new operator splitting algorithm for distributed Nash equilibrium seeking under stochastic uncertainty, featuring relaxation and inertial effects. Our work is inspired by recent deterministic operator splitting methods, designed for solving structured monotone inclusion problems. 
The algorithm is derived from a forward-backward-forward scheme for solving structured monotone inclusion problems featuring a Lipschitz continuous and monotone game operator. To the best of our knowledge, this is the first distributed (generalized) Nash equilibrium seeking algorithm featuring acceleration techniques in stochastic Nash games without assuming cocoercivity. Numerical examples illustrate the effect of inertia and relaxation on the performance of our proposed algorithm.
\end{abstract}

%*************************************************************
%*****    BODY TEXT
%*************************************************************
\renewcommand{\sharp}{\gamma}
\acresetall
\allowdisplaybreaks

\section{Introduction}
\label{sec:intro}
A stochastic generalized Nash equilibrium problem (SGNEP) describes a subclass of competitive multi-agent optimization problems in which local unilateral minimization of an agent-specific expectation-valued cost function subject to system-wide shared coupling constraints. 
Due to the presence of the uncertainty and the shared constraints, computing a SGNE is generally rather challenging. However, these problems have recently received the attention of the system and control community, especially because of their applicability \cite{kulkarni2012,ravat2011,MerStaIFAC19,yu2017,YeHu16} to relevant problems in the engineering sciences. An important class of models in this context is that of networked Cournot games with market capacity constraints and uncertainty in demand and capacity
~\cite{kannan13addressing}. Instances of these models arise in transportation systems,
where the drivers' perception of travel-time is a possible source of uncertainty~\cite{watling2006}, electricity markets where companies dispatch electricity
without \us{an a priori} knowledge of actual demand~\cite{henrion2007}, and natural gas markets where the companies participate in a bounded capacity market~\cite{abada2013}.

Typically, two key concerns arise in any attempt to deal with uncertainty.
First, often the distribution of the random noise is not known to the
agent, thus making the computation of the cost function impossible. Second,
even if the distribution of the stochastic uncertainty is known or predictable
from, say, historical data, a key complication arises when trying to compute the
expected value (and its gradients). Costly simulation-based integration
techniques required employment each time an agent is asked to compute its
decision variable, imposing significant computational burden on each
agent. A versatile alternative to such approaches is provided
by stochastic approximation theory (SA), under which agents' draw fresh samples
at each iteration. 

With the aim of boosting the performance of distributed Nash equilibrium
seeking algorithms, Yi and Pavel~\cite{yi2019} introduced a preconditioned
forward-backward splitting with inertial effects in a completely deterministic
environment where agents receive perfect feedback information. The
possibility of including inertia in the basic forward-backward scheme has received some attention in the field already
before (see e.g. \cite{alvarez2001,attouch2019,attouch2020,lorenz2015}). The common
motivation of all these contributions is to exploit momentum effects to
accelerate \us{the} numerical schemes, inspired by Nesterov's accelerated
gradient method~\cite{Nes83} for convex optimization. However, in the context
of distributed \us{computation of} Nash equilibria, the role of inertial and
acceleration effects is not well understood. This applies in particular to
situations where the game data are subject to stochastic uncertainty so that
agents have only noisy information available in their decision-making process.
Even in the most general problem where one's aim is to solve a stochastic
monotone inclusion~\cite{Bia16,rosasco2016in,rosasco2016,CuiSha20}, standard
acceleration techniques have not received much attention. Our aim is to shed
some light on this highly understudied question and prove some interesting
properties about accelerated game dynamics. This paper departs from recent
progress made in the field of splitting algorithms for stochastic variational
problems, summarized in \cite{KanSha19,SFBF} and \cite{CuiSha20}, which contain
new asymptotic and non-asymptotic results on stochastic sampling-based
algorithms under weaker hypothesis than usually assumed in the computational
game theory literature. In these seminal contributions, stochastic versions of
Tseng's modified extragradient (hitherto forward-backward-forward) algorithm
\cite{Tse00,BauCom16} have been introduced. The importance of this alternative
splitting technique in the context of distributed Nash equilibrium seeking has
been emphasized in \cite{franci2019fbf}. This work extends all these seminal
contributions via an explicit study of the effects of acceleration parameters.
The numerical scheme presented in this paper is provably convergent (in an
almost sure sense) without assuming co-coercivity of the game operator, and can
be implemented via a disciplined mini-batch stochastic approximation
technology, distributed over a network of competing agents. The main result of
this work gives a precise set of parameter sequences ensuring convergence of
the game play to the set of \emph{variational equilibria}, an important subset of generalized Nash equilibria with a clear economic interpretation
\cite{KulSha12}. Our work extends the recent results reported in
\cite{franci2020fbtac}, reliant on forward-backward splitting ideas, and
thus require co-coercivity, as well as the stochastic extragradient method
introduced in \cite{IusJofOliTho17}, where no joint coupling constraints are
considered.
%
%
%When compared with the state-of-the-art solvers, the relaxed-inertial forward-backward-forward algorithm we propose is a significant extension of existing solution techniques for stochastic Nash games. Specifically, it only requires (maximal) monotonicity and Lipschitz continuity of the operator. This is a major improvement compared to the existing distributed Nash seeking algorithms in the stochastic regime that rely on the forward-backward splitting technique, and thus need to impose cocoercivity of the mean single-valued operator \cite{yi2019,franci2020fbtac}. Our approach builds upon, and significantly extends, recent contributions to the field stochastic variational inequalities and generalized equations \cite{bot2020,CuiSha20}. In particular, it generalizes the multi-agent version of the stochastic extragradient method introduced in \cite{IusJofOliTho17} where no joint coupling constraints are studied. It also generalizes the recent contribution \cite{franci2019fbf} by introducing stochastic uncertainty and exploiting stochastic approximation approaches.

\subsection{Basic Notation}

$\RR$ denotes the set of real numbers and $\bar\RR=\RR\cup\{+\infty\}$.
$\langle\cdot,\cdot\rangle:\RR^n\times\RR^n\to\RR$ denotes the standard inner product and $\norm{\cdot}$ represents the associated Euclidean norm. We indicate a (symmetric and) positive definite matrix $A$, i.e., $x^\top Ax>0$, with $A\succ0$. Given a matrix $\Phi\succ0$, we define the $\Phi$-induced inner product as $\langle x, y\rangle_{\Phi}=\langle \Phi x, y\rangle$ and the norm as $\norm{x}_{\Phi}=\sqrt{\langle \Phi x, x\rangle}$. 
$A\otimes B$ indicates the Kronecker product between matrices $A$ and $B$. ${\bf{0}}_m$ (${\bf{1}}_m$) indicates the vector with $m$ entries all equal to $0$ ($1$). Given $x_{1}, \ldots, x_{N} \in \RR^{n}$, $\boldsymbol{x} :=\op{col}\left(x_{1}, \dots, x_{N}\right)=\left[x_{1}^{\top}, \dots, x_{N}^{\top}\right]^{\top}.$

Let $T:\RR^{n}\rightrightarrows \RR^n$ be a set-valued operator. The domain of $T$ are defined by $\dom T= \{x\in\RR^n\mid T(x)\neq\emptyset\}.$ The set of zeros of $T$ is $\Zer(T)= \{x\in\RR^n\mid 0\in T(x)\}$. %The set of fixed points of $T$ is $\op{fix}(T):=\{x\in\RR^n\mid x\in T(x)\}$. 
The resolvent of the operator $T$ is $\mathrm{J}_{T}= (\Id+ T)^{-1}$, where $\op{Id}$ indicates the identity operator. 
An operator $T$ is monotone if $\inner{T(x)-T(y),x-y}\geq 0$ and it is Lipschitz continuous if, for some $\beta>0$, $\|T(x)-T(y)\| \leq \beta\|x-y\|$ for all $x,y\in\dom T$. A monotone operator is maximally monotone if its graph is not properly contained in the graph of another monotone operator.

Given a proper, lower semi-continuous, \us{and} convex function $g$, the subdifferential is the operator $\partial g(x):=\{u\in\Omega \mid (\forall y\in\Omega):\langle y-x,u\rangle+g(x)\leq g(y)\}$. The proximal operator is defined as $\prox_{g}(v):=\op{argmin}_{u\in\Omega}\{g(u)+\tfrac{1}{2}\norm{u-v}^{2}_{}\}=\mathrm{J}_{\partial g}(v)$.
$\iota_C$ is the indicator function of the set $C$, i.e., $\iota_C(x)=1$ if $x\in C$ and $\iota_C(x)=0$ otherwise. The set-valued mapping $\mathrm{N}_{C} : \RR^{n} \rightrightarrows \RR^{n}$ denotes the normal cone operator for the the set $C$ , i.e., $\mathrm{N}_{C}(x)=\varnothing$ if $x \notin C,\left\{v \in \RR^{n} | \sup _{z \in C} v^{\top}(z-x) \leq 0\right\}$ otherwise.

All randomness is modeled on a complete probability space $(\Omega,\scrF,\Pr)$, endowed with a filtration $\F=(\scrF_{k})_{k\geq 0}$. 
%
%We use Standing Assumptions for technical conditions that hold throughout the paper.
%
%%%%%%%%%%%%%%%%%%%%%%%%%%%%%%%%%%%%%%%%%
%%%%%%%% GNEP%%%%%%%%%%%%%%%%%%%%%%%%%%%
%%%%%%%%%%%%%%%%%%%%%%%%%%%
\section{Mathematical Setup}
\label{sec:prelims}
\subsection{Generalized Nash equilibrium problems}
We consider a game where each agent $i\in \mc I=\{1,\ldots,N\}$ chooses an action $u_{i}\in\R^{d_i}$. Let $\ubold=\op{col}(u_{1},\ldots,u_{N})$ and $d\equiv\sum_{i=1}^{N}d_{i}$.
Each agent $i$ has a local cost function $\JJ_{i}: \R^{d} \to \bar\RR$ of the form 
\vspace{-.15cm}\begin{equation}\label{eq:theta}
\JJ_{i}(u_{i}, \ubold_{-i})= f_{i}(u_{i},\ubold_{-i}) + g_{i}(u_{i}).
\vspace{-.15cm}\end{equation}
where $\ubold_{-i}=\text{col}(\{u_j\}_{j\neq i})$ is the vector of all decision variables except for $u_i$, and $g_{i}:\RR^{d_i}\to\bar\RR$ is a local idiosyncratic cost function. The function $\JJ_i$ in \eqref{eq:theta} has the typical splitting into smooth and non-smooth parts. %This describes a fairly large class of convex programming problems, when coupled with the next assumption \cite{DvuShtSta21}.
\begin{standassumption}
For each $i\in\mc I$, the function $g_i$ in \eqref{eq:theta} is proper, convex and lower semi-continuous and $\dom(g_{i})=\setU_i\subseteq\RR^{d_i}$ is (nonempty) compact and convex.
\fineass\end{standassumption}

Examples for the nonsmooth part are indicator functions to enforce local constraints, or penalty functions that promote sparsity, or other desirable structure. 

We assume that the function $f_{i}(u_{i},\ubold_{-i})$ depends on the own action $u_{i}$ and a subset of the others actions $\{u_{j}\}_{j\in\scrN_{i}^{A}}$, where the set $\scrN_{i}^{A}\subset\mc I$ is the \emph{interaction neighborhood} of agent $i$. Furthermore, we assume convexity and differentiability, as usual in the generalized Nash equilibrium problem (GNEP) literature \cite{facchineikanzow2007,facchinei2007vi,kulkarni2012}.
\begin{standassumption}
\label{ass:IC}
For each $i \in\mc I$ and for all $\ubold_{-i}$, the function $f_{i}(\cdot,\ubold_{-i})$ in \eqref{eq:theta} is convex and continuously differentiable. 
\fineass\end{standassumption}

We assume that the game displays joint convexity with affine coupling constraints defining the collective feasible set 
\begin{equation}\label{eq:coupling}
\bs{\mc C} = \{\ubold\in\setU\mid\; D\ubold-b \leq \0_{m}\},
\end{equation}
where $\setU=\setU_1\times\dots\times\setU_N$, $D= [D_1\mid\dots\mid D_N]\in\R^{m\times d}$ and $b= \sum_{i=1}^{N}b_{i}\in\R^m$. 
Each matrix $D_i\in\R^{m\times d_i}$ defines how agent $i$ is involved in the coupling constraints. Given the strategies of all other agents $\ubold_{-i}$, the set of feasible actions of player $i$ is defined as \us{the following set-valued map.} 
\vspace{-.15cm}\begin{equation}
\mc C_{i}(\ubold_{-i}) = \{u_i \in \setU_i \mid D_i u_i -b_{i}\leq \sum_{j \neq i}^{N}(b_{j}-D_j u_j)\}.
\vspace{-.15cm}\end{equation}
%The collection of local cost functions $(f_{i},g_{i})$, together with the linear coupling constrains defined by the pair $(D,b)$, defines a \emph{generalized Nash game} $\scrG(\{(f_{i},g_{i})\}_{i\in\mc I},D,b)$. 

\begin{standassumption}\label{ass_X}
The global feasible set $\boldsymbol{\mc C}$ in \eqref{eq:coupling} satisfies Slater's constraint qualification. 
\fineass\end{standassumption}

\us{For $i \in \mc I$, the $i$th agent solves the following parametrized optimization problem.} %  aim of each agent is to solve its local minimization problem
\vspace{-.15cm}\begin{equation}\label{eq:BR}
\forall i\in\mc I: \quad\left\{\begin{array}{cl}
\min\limits_{u_i \in\RR^{d_{i}}} & \JJ_i(u_i, \ubold_{-i}) \\ 
\text { s.t. } & u_{i}\in\mc C_{i}(\ubold_{-i}).
\end{array}\right.
\vspace{-.15cm}\end{equation}
The usual solution concept for the game with coupling constraints in \eqref{eq:BR} is that of \emph{generalized Nash equilibrium} (GNE) \cite{facchinei2007vi,FacPan03}, i.e., an $N$-tuple $\ubold^{\ast}=\textnormal{col}(u_{1}^{\ast},\ldots,u_{N}^{\ast})\in\setU$ such that for all $i\in\mc I$,
\[
\JJ_i(u_{i}^{\ast},\ubold^{\ast}_{-i}) \leq \inf\{ \JJ_i(u_{i},\ubold^{\ast}_{-i}) \, \mid  \, u_{i}\in\mc C_i(\ubold_{-i})\}.
\]

%\begin{definition} (\textit{Generalized Nash equilibrium})
%A collective action profile $\ubold^{\ast}=\textnormal{col}(u_{1}^{\ast},\ldots,u_{N}^{\ast})\in\setU$ is a generalized Nash equilibrium (GNE) of the game in \eqref{eq:BR} if, for all $i\in\mc I$,
%\[
%J_i(u_{i}^{\ast},\ubold^{\ast}_{-i}) \leq \inf\{ J_i(w,\ubold^{\ast}_{-i}) \, \mid  \, w\in\mc X_i(\ubold_{-i})\}.
%\]
%\end{definition}

%\subsection{Variational Characterization}
Our computational approach for solving the GNEP in \eqref{eq:BR} makes use of the Karush-Kuhn-Tucker (KKT) conditions characterizing the unilateral optimization of the agents. To achieve a numerically tractable framework, we impose some conditions on the model concerning the monotonicity and Lipschitz continuity of the mapping that collects the local pseudogradients of the agents.
\begin{standassumption}
\label{ass:GM}
The pseudogradient mapping 
\begin{equation}\label{eq:F}
F(\ubold)= \mathrm{col}\left( \nabla_{u_1}f_{1}(\ubold),\ldots,\nabla_{u_N}f_{N}(\ubold)\right)
\end{equation}
is monotone %on $\R^{d}$, i.e., for all $\ubold,\tilde{\ubold}\in\R^{d}$, $\inner{F(\ubold)-F(\tilde{\ubold}),\ubold-\tilde{\ubold}}\geq 0.$
%\vspace{-.15cm}\end{equation*} 
and $\ell$-Lipschitz continuous.%, $\beta > 0$, i.e., for all $\ubold,\tilde{\ubold}\in\R^{d}$,\vspace{-.15cm}
%\vspace{-.15cm}\begin{equation*}
%$\norm{F(\ubold)-F(\tilde{\ubold})} \leq \tfrac{1}{\beta}\norm{\ubold-\tilde{\ubold}}.$%\vspace{-.3cm}
%\vspace{-.15cm}\end{equation*}
\fineass\end{standassumption}
%%%%%%%%

The KKT conditions corresponding to the game in \eqref{eq:BR} are necessary and sufficient for characterizing a \us{tuple of strategies to be a GNE}. Among all possible GNEs of the game, we focus on the computation of \emph{variational equilibria} (v-GNE), i.e. a GNE in which all agents share consensus on the dual variables \cite[Theorem 3.1]{facchinei2007vi}, \cite[Theorem 3.1]{auslender2000} which is, in turn, a solution of the variational system
\vspace{-.15cm}\begin{equation}\label{KKT_VI}
\forall i\in\mc I:\begin{cases}
\0_{d_i}\in \nabla_{u_{i}}f_{i}(u^{\ast}_{i},\ubold^{\ast}_{-i})+\partial g_{i}(u^{*}_{i})+D_{i}^{\top}\lambda^{\ast}\\
\0_{m}\in \op{N}_{\R^{m}_{\geq0}}(\lambda^{\ast})-(D\ubold^{\ast}-b).\\
\end{cases}
\vspace{-.15cm}\end{equation}
for some $\lambda^*\in\RR^M_{\geq0}$.
%%%%%%%%%%%%%%%%%%%
\subsection{Distributed \us{GNE} via operator splitting}
A key challenge one faces in any computational approach in Nash equilibrium problems is to resolve the question how players access the decision variables of the other agents. \us{An} attractive approach for resolving this issue is the distributed operator splitting approach pioneered in \cite{yi2019}.

We allow each agent to have information on his own local problem data only, i.e., $\JJ_i$, $\setU_i$, $D_{i}$ and $b_{i}$. Moreover, each agent $i$ controls its local decision $u_{i}$ and a local copy $\lambda_{i}\in\R^{m}_{\geq0}$ of dual variables, as well as a local auxiliary variable $\mu_{i}\in\R^{m}$ used to enforce consensus of the dual variables. To reach such consensus, we let the agents exchange information via an undirected weighted communication graph represented by its weighted adjacency matrix $\bW = [w_{i,j}]\in\R^{N\times N}$. We assume $w_{ij}>0$ iff $(i,j)$ is an edge in the communication graph. The set of neighbors of agent $i$ in the communication graph is $\scrN^{\lambda}_{i}=\{j\mid w_{i,j}>0\}$.
%\vspace{.15cm}
\begin{standassumption}\label{ass:graph}
The adjacency matrix $\bW$ of the communication graph is symmetric and irreducible.
\fineass\end{standassumption}
Let us define the weighted Laplacian as $\bL=\diag\left\{(\bW\1_{N})_{1}, \dots, (\bW\1_{N})_{N}\right\}-\bW$. It holds that $\bL^{\top}=\bL$, $\op{null}(\bL)=\{a\bs 1_N,a\in\RR\}$ and that, given Standing Assumption \ref{ass:graph}, $\bL$ is positive semi-definite with real and distinct eigenvalues $0=s_{1}<s_{2}\leq \ldots \leq s_{N}$.
Moreover, given the maximum (weighted) degree of the graph, $\Delta:=\max_{i\in\mc I}(\bW\1_{N})_{i}$, it holds that $\Delta \leq s_{N}\leq 2\Delta$. Denoting by $\kappa= \abs{\bL}$, it holds that $\kappa\leq 2\Delta$ \cite{godsil2013}.  We define the tensorized Laplacian as the matrix $\bar{\bL}=\bL\otimes \bI_{m}$. We set $\bar{b}=(b_1,\ldots,b_{N})^{\top}$, $\ubold=\text{col}(u_{1},\ldots,u_{N})$ and similarly $\mubold$ and $\lambdabold$. As the state variable, we consider the triple $\xbold=(\ubold,\mubold,\lambdabold)\in\setX:=\R^{n}\times\R^{mN}\times\R^{mN}$ and endow $\setX$ with the product topology.
Let $\bD= \diag\{D_1,\ldots,D_N\}$. Then, we define the maximally monotone operators 
\begin{align}
V(\xbold)&= \left[\begin{array}{c} F(\ubold)+\bD^{\top}\lambdabold\\
\bar{\bL}\lambdabold\\
\bar{b}+\bar{\bL}(\lambdabold-\mubold)-\bar{\bD}\ubold\end{array}\right],\\
T(\xbold)&= G(\ubold)\times \{\0_{Nm}\}\times \op{N}_{\R^{mN}_{\geq 0}}(\lambdabold),
\end{align}
where $G(\ubold)= \partial g_{1}(u_{1})\times\cdots\times\partial g_{N}(u_{N})$. 
Let us summarize the properties of the operators above.
\begin{lemma}\label{lemma_op}
The following statements hold:
\begin{itemize}
\item[(i)] $V:\setX\to\setX$ is maximally monotone and $\ell_{V}=( \ell+2\kappa+\abs{\bD})$-Lipschitz continuous.
\item[(ii)] $T:\setX \rightrightarrows \setX$ is maximally monotone.
\end{itemize}
\end{lemma}
\begin{proof}
(i) We can split the operator $V$ into the parts $V_{1}(\xbold)= \text{col}(F(\ubold),{\bf{0}}_{mN},\bar b+\bar{\bL}\lambdabold)$ and $V_{2}(\xbold)= \text{col}(\bD^{\top}\lambdabold, \bar{\bL}\lambdabold,-\bD\ubold-\bar{\bL}\mubold)$, which are maximally monotone by \cite[Prop. 20.23]{BauCom16}, \cite[Cor. 20.28]{BauCom16}. Furthermore, similarly to \cite[Lemma 1]{franci2019fbf} $V_{1}$ is $\ell_{1}=(\ell+\kappa)$-Lipschitz continuous and $V_{2}$ is $\ell_{2}=(|\bD|+\kappa)$-Lipschitz continuous. Hence, $V$ is $\ell_{1}+\ell_{2}=\ell_{V}$-Lipschitz continuous.

(ii) It follows from \cite[Lemma 5]{yi2019}, \cite[Lemma 1]{franci2019fbf}.
\end{proof}
%\begin{proof}
%(i) We can split the operator $V$ into the parts $V_{1}(\xbold)= \text{col}(F(\ubold),{\bf{0}}_{mN},\bar b+\bar{\bL}\lambdabold)$, and $
%V_{2}(\xbold)= \text{col}(\bD^{\top}\lambdabold, \bar{\bL}\lambdabold,-\bD\ubold-\bar{\bL}\mubold)$. Both are easily seen to be maximally monotone (\cite[Prop. 20.23]{BauCom16}, and \cite[Cor. 20.28]{BauCom16}). Furthermore, given $\kappa= \abs{\bar{\bL}}$ and $\xbold:=\text{col}(\ubold,\mubold,\lambdabold)$, it holds that 
%\begin{align*}
%\norm{ V_{1}(\xbold)-V_{1}(\xbold')}\leq& \norm{F(\ubold)-F(\ubold')}+\norm{\bar{\bL}(\mubold-\mubold')}\\
%\leq&( \ell+\kappa)\left(\norm{\ubold-\ubold'}+\norm{\mubold-\mubold'}\right),
%\end{align*}
%showing that $V_{1}$ is $\ell_{1}:=(\ell+\kappa)$-Lipschitz continuous. Similarly, it can be shown that $V_{1}$ is $\ell_{2}=(|\bD|+\kappa)$-Lipschitz continuous. Hence, $V=V_{1}+V_{2}$ is $\ell_{1}+\ell_{2}\equiv\ell_{V}$-Lipschitz continuous.\\
%(ii) Consider the function $\sum_{i=1}^{N}g_{i}(u_{i})+\delta_{\R^{mN}_{\geq 0}}(\lambdabold)$. Then $T$ is seen to be the subdifferential of this function. Since the sum of convex proper lower semi-continuous function has a maximally monotone subdifferential, the  claim follows.
%\end{proof}
The splitting $V+T$ encodes a distributed version of the KKT conditions for
v-GNE \eqref{eq:BR}. In particular, it can be shown that the zeros of the
maximally monotone inclusion $V+T$ are in correspondence with variational
equilibria of the Nash game.  \begin{proposition}
The set $\zer(V+T)$ coincides with the set of v-GNE of the game satisfying the KKT conditions in \eqref{KKT_VI}.
\end{proposition}
\begin{proof}
This follows from~\cite[Thm. 2]{yi2019}~or~\cite[Lemma 3]{franci2019fbf}.
\end{proof}

% we have that the zeros of $V+T$ are the v-GNE of the game in \eqref{eq:BR}. 
%More precisely, let us define the projection $\Pi:\setX\to\R^{d}$, given by $\Pi(\ubold,\mubold,\lambdabold)=\ubold$. 
%\begin{proposition}
%The set $\Pi(\zer(V+T))$ corresponds to variational equilibria of the game $\scrG(\{(f_{i},g_{i})\}_{i\in\mc I},D,b)$.
%\end{proposition}

\subsection{Stochastic GNEPs}
Stochastic uncertainty affecting the decision problem of agent $i$ is modeled by a random variable $\xi_{i}:\Omega\to\Xi_{i}$, where $\Xi_{i}\subset\R^{q_{i}}$ is a given measurable set. We assume that the uncertainty enters the model in the smooth part of the agents' optimization problem, i.e., for each $i\in\mc I$
\begin{equation}\label{eq:Ef}
\JJ_i(u_i,\ubold_{-i})=\Ex[\hat{f}_{i}(\ubold,\xi_{i})]+g_i(u_i).%=\int_{\Omega}\hat{f}_{i}(\ubold,\xi_{i}(\omega))\dif\Pr(\omega).
\end{equation}
%Under this assumption, agent $i$'s objective function $J_i$ appears as a non-smooth composite function which is expectation-valued. 
It follows that the local optimization problems in \eqref{eq:BR} describes a stochastic programming problem, parameterized by the decisions of the opponents $\ubold_{-i}$. 
%Solving such problems practically is a key question in stochastic programming. Two key concerns arise in any attempt to solve problem \eqref{eq:BR}: (i) In many cases the distribution of the random noise term $\xi_{i}$ is not known to the agent, thus making the computation of the cost function impossible; (ii) even if the law of the stochastic uncertainty is known or easy to predict from historical data, a key complication arises when solving the integral in \eqref{eq:Ef}. Costly simulation-based methods integration techniques would have to be employed each time an agent is asked to compute its current best-response, which imposes significant computational burdens on behalf of each agent. A versatile alternative to such simulation based approaches is provided by stochastic approximation theory. 

Let $k\in\N$ denote the iteration count of our computational procedure. We assume that at each round $k$ agent $i$ is able to generate a random sample $\xi_{i,k}=(\xi_{i,k}^{(t)})_{t=1}^{S_k}$, consisting of i.i.d copies of the random element $\xi_{i}$. This sample is used to construct an agent-specific gradient estimator of the form 
\begin{equation}\label{eq:hatF}
\hat{F}_{i,k}(\ubold,\xi_{i,k})=\frac{1}{S_k}\sum_{t=1}^{S_k}\nabla_{u_{i}}\hat{f}_{i}(\ubold,\xi^{(t)}_{i,k}),
\end{equation}
where $S_k\geq 1$ is the size of the data sample. \eqref{eq:hatF} is an example of a mini-batch estimator, which interpolates between cheap sampling and precision. The degree of precision is regulated via the batch size sequence $\{S_{k}\}_{k}$. Dynamically adjusting the size of the batch simulates an online variance reduction mechanism, which plays a key role in our convergence analysis of the distributed operator splitting algorithm to come. Mini-batch samples are prominent in simulation-based optimization, where taking repeated samples of stochastic gradients is computationally cheap \cite{Byrd2012,IusJofOliTho17,SFBF,CuiSha20,lei2018distributed}. 
\begin{standassumption}\label{ass:batch}
The batch size $(S_k)_{k\geq 1}$ is increasing and such that $\sum_{k\in\NN}\frac{1}{S_k}<\infty.$%$(1/S_k)_{k\geq 1}$ is a nonnegative summable sequence for all $i\in\mc I$. 
\fineass\end{standassumption}
Under the prevailing i.i.d. assumption, it holds true that $\Ex[\hat{F}_{i,k}(\ubold,\xi_{i,k})|\ubold]=F_{i}(\ubold)$ for all $i\in\mc I$ and all $\ubold\in\R^{d}$. Hence, the random variable \eqref{eq:hatF} is an \emph{unbiased} estimator of the individual payoff gradient at each action profile $\ubold$. Upon defining the random operator
\begin{equation}
\hat{V}_{k}(\xbold,\xi_{k})=  \left[\begin{array}{c} \hat{F}_{k}(\ubold,\xi_{k})+\bD^{\top}\lambdabold\\
\bar{\bL}\lambdabold\\
\bar{b}+\bar{\bL}(\lambdabold-\mubold)-\bD\ubold\end{array}\right],\\
\end{equation}
with $\xi_k=\op{col}(\xi_{i,k})_{i\in\mc I}$, we see that $\Ex[\hat{V}_{k}(\xbold,\xi_{k})|\xbold]=V(\xbold)$ for all $\xbold=(\ubold,\mubold,\lambdabold)\in\setX$.

Fundamental to the analysis of stochastic approximation algorithms is the control of the stochastic error, defined for all $k\in\NN$ as
\begin{equation}\label{eq:eps}
\eps_{k}(\xbold,\xi_{k})= \hat{V}_{k}(\xbold,\xi_{k})-V(\xbold)\quad\forall \xbold\in\setX.
\end{equation}

\begin{standassumption}\label{ass:variance}
    \us{There exists $\sigma > 0$ such that for all} $k\in\NN$,  the stochastic error is such that \us{the following hold $\mathbb{P}$-a.s. .}
\vspace{.05cm}
\begin{align}
\Ex_{\Pr}[\eps_{k}(\xbold,\xi_{k})\mid\xbold]=0, \label{eq:epszero}\text{ and }\\
\Ex_{\Pr}[\norm{\eps_{k}(\xbold,\xi_{k})}^{2}\mid\xbold]\leq\frac{\sigma^{2}}{S_{k}}. \label{eq:epsvariance}
\end{align}
%for some $\sigma>0$.
\vspace{.05cm}
\fineass\end{standassumption}
\begin{remark}
    Assumption \ref{ass:variance} is rather mild and standard in stochastic optimization \cite{bot2020,franci2020fbtac} while Condition \eqref{eq:epszero} means that the random operator $\hat{V}_{k}(\xbold,\xi_{k})$ is \us{a conditionally} unbiased estimator of $V(\xbold)$. \us{Note that} \eqref{eq:epsvariance} can be satisfied if the sequence of martingale difference errors $\hat{F}_{k}(\ubold,\xi)-F(\ubold)$ satisfies a uniform variance bound \cite{bot2020,franci2020fbtac}. 
\end{remark}

%%%%%%%%%%%%%%%%%%%%%%%%%%%%%%%%%%%%%%%%%%%%%%
%%%%%%ALGORITHM%%%%%%%%%%%%%

\section{A distributed algorithm}

\begin{algorithm}[t]
\caption{Distributed Relaxed Inertial Stochastic Forward Backward Forward (RISFBF)}\label{alg:RISFBF}
Initialization: $u_{i,0} \in\R^{d_{i}} , \lambda_{i,0}\in \R_{\geq0}^{m},$ and $\mu_{i,0} \in \R^{m} .$\\
Iteration $k$: Agent $i$\\
($1$) Perform inertia step: 
$$\begin{aligned}
u^{in}_{i,k}&=u_{i,k}+\alpha(u_{i,k}-u_{i,k-1})\\
\mu^{in}_{i,k}&=\mu_{i,k}+\alpha(\mu_{i,k}-\mu_{i,k-1})\\
\lambda^{in}_{i,k}&=\lambda_{i,k}+\alpha(\lambda_{i,k}-\lambda_{i,k-1}).
\end{aligned}
$$
($2$) Receives $u^{in}_{j,k}$ for $j \in \scrN_{i}^{A}$, $ \lambda^{in}_{j,k}$ and $\mu^{in}_{j,k}$ for $j \in \scrN_{i}^{\lambda}$ and update %$(u^{md}_{i,k},\mu^{md}_{i,k},\lambda^{md}_{i,k})$ as in Step (2) of Algorithm \ref{alg:SFBF}. 
$$\begin{aligned}
&u^{md}_{i,k}=\prox_{\gamma_{i}g_{i}}[u^{in}_{i,k}-\gamma_{i}(\hat{F}_{i,k}(\ubold^{in}_{k},\xi_{k})+D_{i}^{\top} \lambda_{i,k})]\\
&\mu^{md}_{i,k}=\mu^{in}_{i,k}+\sigma_{i} \sum\nolimits_{j} w_{i,j}(\lambda^{in}_{j,k}-\lambda^{md}_{i,k})\\
&\lambda^{md}_{i,k}=\Pi_{\R^{m}_{\geq 0}}\{\lambda^{in}_{i,k}+\tau_{i}(D_{i}u^{in}_{i,k}-b_{i})\\
&\quad\quad+\tau\sum\nolimits_{j} w_{i,j}[(\mu^{in}_{i,k}-\mu^{in}_{j,k})-(\lambda^{in}_{i,k}-\lambda^{in}_{j,k})]\}
\end{aligned}$$

($3$) Receives $u^{md}_{j,k}$ for $j \in \scrN_{i}^{A}$, $ \lambda^{md}_{j,k}$ and $\mu^{md}_{j,k}$ for $j \in \scrN_{i}^{\lambda}$ and performs a relaxation step:
$$\begin{aligned}
&u_{i,k+1}=(1-\rho_{k})u^{in}_{i,k}+\rho_{k}[u^{md}_{i,k}+\gamma_{i}(\hat{F}_{i,k}(\ubold^{in}_{k},\xi_{i,k})+\\
&\quad\quad\quad-\hat{F}_{i,k}(\ubold^{md}_{k},\eta_{i,k}))+\gamma_{i}D_{i}^{\top} (\lambda^{in}_{i,k}-\lambda^{md}_{i,k})\\
&\mu_{i,k+1}=(1-\rho_{k})\mu_{i,k}^{in}+\rho_{k}\{\mu^{md}_{i,k}+\\
&\quad \sigma_{i} \sum\nolimits_{j} w_{i,j}[(\lambda^{in}_{i,k}-\lambda^{in}_{j,k})-(\lambda^{md}_{i,k}-\lambda^{md}_{j,k})]\}\\
&\lambda_{i,k+1}=(1-\rho_{k})\lambda^{in}_{i,k}+\rho_{k}\{[\lambda^{md}_{i,k}+\tau_{i}D_i(u^{in}_{i,k}-u^{md}_{i,k})\\
&\quad\quad\quad-\tau_i\sum\nolimits_{j \in \mathcal{N}_{i}^{\lambda}} w_{i,j}[(\mu^{md}_{i,k}-\mu^{md}_{j,k})-(\mu^{in}_{i,k}-\mu^{in}_{j,k})]\\
&\quad\quad\quad+\tau_i\sum\nolimits_{j \in \mathcal{N}_{i}^{\lambda}} w_{i,j}[(\lambda^{in}_{i,k}-\lambda^{in}_{j,k})-(\lambda^{md}_{i,k}-\lambda^{md}_{j,k})]\}
\end{aligned}$$
\end{algorithm}

% 1
%The distributed iterations we propose represent a variant of the forward-backward-forward (FBF) algorithm \cite{bot2020}.
%2
With the \us{intent of boosting} the convergence of distributed Nash seeking algorithms, we propose a relaxed inertial forward-backward-forward algorithm (RISFBF), presented in Algorithm \ref{alg:RISFBF}. Using operator-theoretic notation, the numerical scheme can be restated compactly as 
\begin{equation}\label{eq:alg2}
\left\{\begin{array}{l}
Z_{k}=X_{k}+\alpha_{k}(X_{k}-X_{k-1}),\\
Y_{k}=\mathrm{J}_{\Psi^{-1}T}(Z_{k}-\Psi^{-1}\hat{V}_{k}(Z_{k},\xi_{k})),\\
X_{k+1}=(1-\rho_{k})Z_{k}+\\
 \qquad\rho_{k}[Y_{k}-\Psi^{-1}(\hat{V}_{k}(Y_{k},\eta_{k})-\hat{V}_{k}(Z_{k},\xi_{k}))],
\end{array}\right.
\end{equation}
% A
where $Z_{k}=(\ubold^{in}_{k},\mubold^{in}_{k},\lambdabold_{k}^{in})$, $X_{k}= (\ubold_{k},\mubold_{k},\lambdabold_{k})$ and $Y_{k}=(\ubold^{md}_{k},\mubold^{md}_{k},\lambdabold_{k}^{md})$. 
%B
The random sequence $\eta_{k}= (\eta_{i,k})_{i\in\mc I}$ is another i.i.d. random sample, generated independently by each agent  after the first updating step in Algorithm \ref{alg:RISFBF} is completed. The preconditioning matrix 
\begin{equation}\label{eq:Psi}
\Psi= \diag(\gammabold^{-1},\sigmabold^{-1},\taubold^{-1})
\end{equation}
collects all agent-specific step sizes, so that $\gammabold=\diag\{\gamma_{1}\bI_{d_{1}},\ldots,\gamma_{N}\bI_{d_{N}}\}$ is a block-diagonal matrix with $\gamma_{i}>0$ (analogously, $\sigmabold$ and $\taubold$). 

% 3A -1
The iterations involve first an inertial step in the primal-dual space. Then, there is
% 3A -3
a proximal step corresponding to a gradient-based update given the stochastic estimate of the pseudogradient and the local estimate of the dual variable, followed by a consensus-enforcing estimate merging the values of the dual variables of neighboring agents, and a dual update in the spirit of Lagrangian methods. 
%3B- 4
The last step is a weighted average between the inertial iterate $Z_{k}$ and the forward update $Y_{k}$.
%3A - 5
The algorithm is distributed and involves communication only in terms of dual variables. This fact makes the scheme very attractive for decentralized implementations in large \us{networked game-theoretic settings}.

%C
Standing Assumption \ref{ass:GM} and Lemma \ref{lemma_op} \us{imply} that $V$ is monotone and $\ell_{V,\Psi}=\ell_{V}/\lambda_{\min}(\Psi)$-Lipschitz continuous in the $\Psi$-induced norm \cite{franci2019fbf}.

\begin{theorem}\label{th:RISFBF}
\ssc{Suppose $\nu$ is a positive scalar where $0<\nu<1$.} Let $\lambda_{\ssc{\min}}(\ssc{\Psi})\in\left(0,\ssc{\tfrac{1-\nu}{2\ell_{V}}}\right)$, $0<\alpha_k\le\bar{\alpha}<1$ and 
$\rho_k = {\tfrac{\ssc{(3-\nu)}(1-\bar{\alpha})^2}{\ssc{2}(2\alpha_k^2-\alpha_k+1)(1+\ell_{V,\Psi})}}$. Then, the sequence $(\ubold_{k})_{k\geq 1}$ generated by Algorithm \ref{alg:RISFBF} converges almost surely to a v-GNE of the game in \eqref{eq:BR}.
\end{theorem}
\begin{proof}
See Section \ref{sec:RISFBF}.
\end{proof}

\begin{remark}
The classic SFBF algorithm \cite{SFBF} is obtained as a special case by taking $\alpha=0$ and $\rho_k=1$. Under noise-free feedback, this scheme would coincide with the operator-splitting approach of \cite{franci2019fbf}.
\end{remark}

%%%%%%%%%%%%%%%%%%%%%%%%%%%%%%%%%%%%%%%%%%%%%%
%%%%%%Applications%%%%%%%%%%%%%
\section{Numerical Results}
\label{sec:Applications}

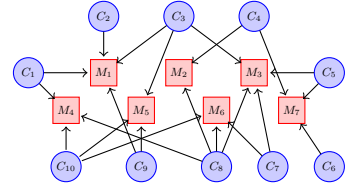
\begin{figure}
\begin{center}
\begin{tikzpicture}[scale=1,every node/.style={circle,draw=black,scale=.5,minimum size=.8cm}]
 
%\draw[thin,gray,scale=1] (-1,-4) grid (4,1);
%\node[draw=red, fill=red](0,0){};

\node[rectangle,draw=red,fill=red!20,minimum size=.7cm](m1) at (.5,-1){};
\node[rectangle,draw=red,fill=red!20,minimum size=.7cm](m2) at (1.5,-1){};
\node[rectangle,draw=red,fill=red!20,minimum size=.7cm](m3) at (2.5,-1){};
\node[rectangle,draw=red,fill=red!20,minimum size=.7cm](m4) at (0,-1.5){};
\node[rectangle,draw=red,fill=red!20,minimum size=.7cm](m5) at (1,-1.5){};
\node[rectangle,draw=red,fill=red!20,minimum size=.7cm](m6) at (2,-1.5){};
\node[rectangle,draw=red,fill=red!20,minimum size=.7cm](m7) at (3,-1.5){};

\node[circle,draw=blue,fill=blue!20](c1) at (-.5,-1){};
\node[circle,draw=blue,fill=blue!20](c2) at (.5,-.25){};
\node[circle,draw=blue,fill=blue!20](c3) at (1.5,-.25){};
\node[circle,draw=blue,fill=blue!20](c4) at (2.5,-.25){};
\node[circle,draw=blue,fill=blue!20](c5) at (3.5,-1){};
\node[circle,draw=blue,fill=blue!20](c6) at (3.5,-2.25){};
\node[circle,draw=blue,fill=blue!20](c7) at (2.75,-2.25){};
\node[circle,draw=blue,fill=blue!20](c8) at (2,-2.25){};
\node[circle,draw=blue,fill=blue!20](c9) at (1,-2.25){};
\node[circle,draw=blue,fill=blue!20](c10) at (0,-2.25){};

\node[draw=none,fill=none](m1) at (.5,-1){\small$M_1$};
\node[draw=none,fill=none](m2) at (1.5,-1){\small$M_2$};
\node[draw=none,fill=none](m3) at (2.5,-1){\small$M_3$};
\node[draw=none,fill=none](m4) at (0,-1.5){\small$M_4$};
\node[draw=none,fill=none](m5) at (1,-1.5){\small$M_5$};
\node[draw=none,fill=none](m6) at (2,-1.5){\small$M_6$};
\node[draw=none,fill=none](m7) at (3,-1.5){\small$M_7$};

\node[draw=none,fill=none](c1) at (-.5,-1){\small$C_1$};
\node[draw=none,fill=none](c2) at (.5,-.25){\small$C_2$};
\node[draw=none,fill=none](c3) at (1.5,-.25){\small$C_3$};
\node[draw=none,fill=none](c4) at (2.5,-.25){\small$C_4$};
\node[draw=none,fill=none](c5) at (3.5,-1){\small$C_5$};
\node[draw=none,fill=none](c6) at (3.5,-2.25){\small$C_6$};
\node[draw=none,fill=none](c7) at (2.75,-2.25){\small$C_7$};
\node[draw=none,fill=none](c8) at (2,-2.25){\small$C_8$};
\node[draw=none,fill=none](c9) at (1,-2.25){\small$C_9$};
\node[draw=none,fill=none](c10) at (0,-2.25){\small$C_{10}$};

\foreach \from/\to in
{c1/m1,c1/m4,c2/m1,c3/m1,c3/m5,c3/m3,c4/m2,c4/m7,c5/m3,c5/m7,c6/m7,c7/m3,c7/m6,c8/m3,c8/m6,c8/m2,c8/m4,c9/m5,c9/m1,c10/m4,c10/m5,c10/m6}
\draw[-to] (\from) -- (\to); 

\end{tikzpicture}
\end{center}
\caption{Networked Cournot game: an edge from $C_i$ to $M_j$ means that company $i$ sells energy in market $j$.}\label{fig_market}
\end{figure}

%\begin{figure}[t]
%\begin{center}
%\includegraphics[width=\columnwidth]{Figs/RISFBFvsSFBFvsSFB_01.eps}
%\end{center}
%\caption{Residual distance of the primal variable form the solution.}\label{fig_comparison}
%\end{figure}
%
%\begin{figure}[t]
%\begin{center}
%\includegraphics[width=\columnwidth]{Figs/RISFB_acc_avg.eps}
%\end{center}
%\caption{Residual distance of the primal variable form the solution for the RISFBF algorithm varying the inertial parameter.}\label{fig_alpha}
%\end{figure}

In this section, we report the results of some numerical simulations to illustrate the improved performance of the RISFBF algorithm (Algorithm \ref{alg:RISFBF}) compared to the classic SFBF \cite{bot2020,franci2019fbf} and with the preconditioned SFB \cite{franci2020fbtac,yi2019}.

Let us consider a networked Cournot problem with market capacity constraints as, for instance, the electricity market or the gas market, inspired by \cite{kannan2012}. We suppose that there are $N=10$ firms selling energy in $m=7$ markets. Not every company sells quantities on each market. Instead, we let $\mathcal{M}_{i}\subseteq\{1,\ldots,m\}$ denote the subset of markets firm $i$ is active on. Each company has a cost of production $c_i(u_i)=c_i^\top u_i$ where $c_i\in\RR^{d_i}$ is chosen according to a truncated normal distribution, i.e., $[c_i]_j=\max(N(2,1),0.6)$. Moreover, each market $j$ has an inverse demand function $P_j(\ubold,\xi)=q_j+p_j(\xi)[S_j(\ubold)]^\sigma$ where $q_j=400$ and $p_j(\xi)$ depends on the unknown random variable, e.g, the overall demand. The values of $p_j(\xi)$ are randomly generated with a normal distribution with mean $0.02$ and bounded variance. The variable $S_j(\ubold)=\sum_{i\in\mc I}[u_i]_j$ couples the actions of the companies and it represents the total energy sold in market $j$. Hence, the cost function of each company is $\JJ_{i}(u_i,\bs u_{-i})=c_i(u_i)-\sum_{j\in\scrM_{i}}\EE[P_j(\ubold,\xi)[u_i]_j]$. The corresponding pseudogradient mapping is monotone, according to \cite[Section 4]{kannan2012}, for $1<\sigma\leq 3$. Therefore, we fix $\sigma=1.2$.
Moreover, we suppose that the companies have a limited production, i.e., $0\leq [u_i]_{j}\leq\theta_{i,j}$ with $\theta_{i,j}=\max(N(250,50),0)$ for $j\in\mathcal{M}_{i}$. This can be incorporated by setting $g_{i}(u_{i})=\sum_{j\in\mathcal{N}_{i}^{A}}\iota_{[0,\theta_{i,j}]}([u_{i}]_{j})$. Similarly, the markets have a bounded capacity $b_j\in[5,10],j=1,\ldots,m,$ and the coupling between the companies can be retrieved from Figure \ref{fig_market}.

\begin{figure}[t]
\begin{center}
\includegraphics[width=\columnwidth]{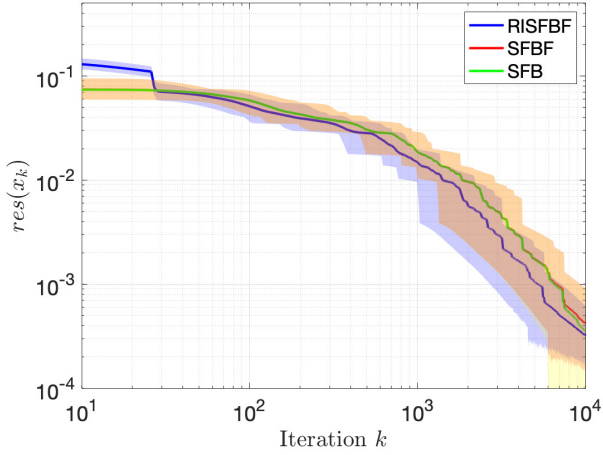}
\end{center}
\caption{Residual distance of the primal variable form the solution.}\label{fig_comparison}
\end{figure}

%\begin{figure}[t]
%\begin{center}
%\includegraphics[width=\columnwidth]{Figs/RISFB_acc_avg.eps}
%\end{center}
%\caption{Residual distance of the primal variable form the solution for the RISFBF algorithm varying the inertial parameter.}\label{fig_alpha}
%\end{figure}

The plot in Figure \ref{fig_comparison} shows the performance, in terms of the residual, of our proposed algorithm in comparison with the SFBF and SFB algorithms. The residual mapping is defined as $\op{res}(x^k)=\norm{x^k-\op{proj}_{\mathcal C}(x^k-F(x^k))}$ and it measures the distance of the primal variable from being a Nash equilibrium. The thick line indicates the average performance and the transparent area is the variability over 10 simulations.
The acceleration parameter is updated according to $\alpha_k=\bar\alpha(1-\frac{1}{k+1})$ with $\bar\alpha=0.1$ and the relaxation parameter is $\rho_k=\frac{(3-\nu)(1-\bar\alpha)^2}{2(2\alpha_k^2-\alpha_k+1)(1+\ell_{V,\Psi})}$, with $\nu=0.01$.
Figure \ref{fig_alpha} shows how the performance changes varying the inertial parameter $\bar\alpha$ while $\rho_k=1$ is fixed. For the sake of comparison, we also include the performance with the same parameters as in Figure \ref{fig_comparison}, the updating rule for $\rho_k$ as in Theorem \ref{th:RISFBF} and the SFBF ($\alpha_k=0$, $\rho_k=0$).

\section{Analysis}
\label{sec:analysis}

\subsection{Preparatory facts}

To simplify the analysis, let us define the random processes $\scrA_{k}:=\hat{V}_{k}(Z_{k},\xi_{k})$ and $\scrB_{k}:=\hat{V}_{k}(Y_{k},\eta_{k})$. Define the sub-sigma algebra $\scrF_{k}:=\sigma(\xbold_{0},\xi_{0},\ldots,\xi_{k-1},\eta_{0},\ldots,\eta_{k-1})$, and $\scrG_{k}:=\sigma\left(\scrF_{k}\cup\sigma(\xi_{k})\right)$. We introduce the centered error processes $U_{k}:=\scrA_{k}-\Ex[\scrA_{k}\mid\scrF_{k}]$ and $W_{k}:=\scrB_{k}-\Ex[\scrB_{k}\mid\scrG_{k}]$. Note that Standing Assumption \ref{ass:variance} implies that $\Ex[U_{k}\mid\scrF_{k}]=\Ex[W_{k}\mid\scrF_{k}]=0$ and that $\left(\Ex[\norm{U_{k}}^{2}_{\Psi^{-1}}\mid\scrF_{k}]\right)_{k\geq 1}$ and $\left(\Ex[\norm{W_{k}}^{2}_{\Psi^{-1}}\mid\scrF_{k}]\right)_{k\geq 1}$ are summable sequences.

Define the residual function for the monotone inclusion as
%\begin{equation}\label{eq:res}
$r_{\Psi}(x)= \norm{x-\mathrm{J}_{\Psi^{-1} T}(x-\Psi^{-1} V(x))}.$
%\end{equation}
For every $\Psi\succ0$, $x\in\zer(V+T)\iff r_{\Psi}(x)=0$.

\begin{lemma}\label{lem:ab}
For $x,y\in\setX$ and $\alpha,\beta\geq 0$ with $\alpha+\beta=1$, it holds that 
%\begin{equation}
$\norm{\alpha x+\beta y}^2 = \alpha\norm{x}^2 + \beta\norm{y}^2- \alpha \beta \norm{x-y}^2.$
%\end{equation}
\end{lemma}

\begin{figure}[t]
\begin{center}
\includegraphics[width=\columnwidth]{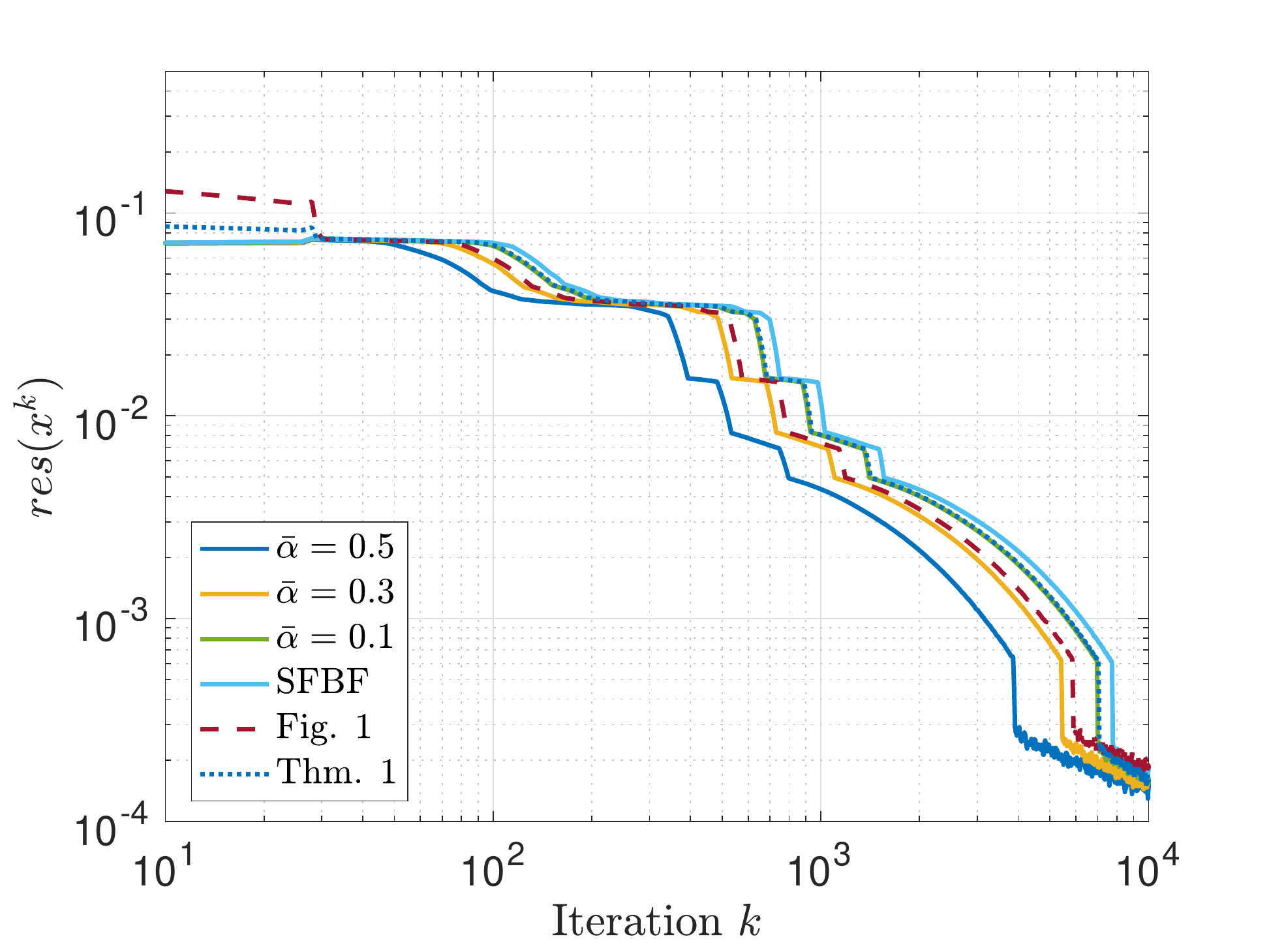}
\end{center}
\caption{Residual distance of the primal variable form the solution for the RISFBF algorithm varying the inertial parameter.}\label{fig_alpha}
\end{figure}

\begin{lemma}[Robbins-Siegmund]\cite[Lemma 11, page 50]{Pol87}.
\label{lem:RS}
Let $(\Omega,\scrF,\F=(\scrF_{k})_{k\geq 0},\Pr)$ be a discrete stochastic basis. 
Let $(\alpha_k)_{k\in\NN}$, $(\theta_k)_{k\in\NN}$, $(\eta_k)_{k\in\NN}$ and $(\chi_k)_{k\in\NN}$ be non-negative processes such that $\sum_k\eta_k<\infty$, $\sum_k\chi_k<\infty$ and let
$$\forall k\in\NN, \quad \EE[\alpha_{k+1}|\mc F_k]+\theta_k\leq (1+\chi_k)\alpha_k+\eta_k \quad a.s.$$
Then $\sum_k \theta_k<\infty$ and $(\alpha_k)_{k\in\NN}$ converges a.s. to a non negative random variable.
\end{lemma}

\begin{lemma}\label{lem:YZG}
%Let $(S_{k})_{k\geq 1}$ be defined as 
%\[
%S_{k}= \left\{\begin{array}{ll} X_{k}& \text{if SFBF is used},\\
%Z_{k} & \text{if RISFBF is used}.
%\end{array}\right.
%\]
%Defining the process $(S_{k})_{k\geq 0}$ in this way, we obtain a uniform description of the update $Y_{k}$ in SFBF and RISFBF, given by $Y_{k}=J_{\Psi^{-1}T}(S_{k}-\Psi^{-1}\scrA_{k})$, where $\scrA_{k}$ defined in \eqref{eq:A}
For all $k\geq 1$ we have 
\begin{equation}\label{eq:YZG}
-\norm{Z_{k}-Y_{k}}^{2}_{\Psi}\leq \norm{U_{k}}^{2}_{\Psi^{-1}}-\tfrac{1}{2}r^{2}_{\Psi}(Z_{k}).
\end{equation}
\end{lemma}
\begin{proof}
By definition 
\begin{align*}
&\tfrac{1}{2}r_{\Psi}^{2}(Z_{k})=\tfrac{1}{2}\norm{Z_{k}-\mathrm{J}_{\Psi^{-1} T}(Z_{k}-\Psi^{-1}V(Z_{k})}^{2}\\
%&=\tfrac{1}{2}\norm{Z_{k}-Y_{k}+J_{\Psi^{-1}T}(Z_{k}-\Psi^{-1}\scrA_{k})-J_{\Psi^{-1} T}(Z_{k}-\Psi^{-1}V(Z_{k})}^{2}\\
&\leq \norm{Z_{k}-Y_{k}}^{2}_{\Psi}\\
&+\norm{\mathrm{J}_{\Psi^{-1}T}(Z_{k}-\Psi^{-1}\scrA_{k})-\mathrm{J}_{\Psi^{-1} T}(Z_{k}-\Psi^{-1}V(Z_{k}))}^{2}_{\Psi}\\
&\leq \norm{Z_{k}-Y_{k}}^{2}_{\Psi}+\norm{U_{k}}^{2}_{\Psi^{-1}}
\end{align*}
where the last inequality uses the non-expansiveness of the resolvent operator $J_{\Psi^{-1}T}$ under the norm $\norm{\cdot}_{\Psi}$. 
\end{proof}

\subsection{Convergence analysis of RISFBF algorithm}
\label{sec:RISFBF}
Define the stochastic processes
\begin{align}
&\Delta M_{k}:= \textstyle{\frac{\ssc{(3-\nu)}\rho_{k}}{1+\ell_{V,\Psi}} \norm{\ce_{k}}^{2}+\textstyle{\ssc{\nu}\rho_{k}}\norm{U_{k}}^{2}_{\Psi^{-1}}},\label{eq:M}\\
&\Delta N_{k}(p):= 2\rho_{k}\inner{W_{k},p-Y_{k}},
\label{eq:N}%\\
%&\ce_{k}= W_{k}-U_{k}.
%\label{eq:e}
\end{align}
with ${\ce_{k}}:= W_{k}-U_{k}$.
We start proving the following fundamental inequality.
\begin{lemma}[Fundamental Recursion]
Fix $p\in\zer(V+T)$ arbitrary, and set $R_{k}= Y_{k}+\Psi^{-1}(\scrA_{k}-\scrB_{k})$. For all $k\geq 0$, it holds true that 
\begin{align*}
&\norm{X_{k+1}-p}^{2}_{\Psi}\leq (1+\alpha_{k})\norm{X_{k}-p}^{2}_{\Psi}-\alpha_{k}\norm{X_{k-1}-p}^{2}_{\Psi}\\
&+\Delta M_{k}+\Delta N_{k}(p)-\textstyle{\frac{\ssc{\nu}\rho_{k}}{\ssc{2}}}r^{2}_{\Psi}(Z_{k})\\
&+\alpha_{k}\norm{X_{k}-X_{k-1}}^{2}_{\Psi}\textstyle{\left(2\alpha_{k}+\frac{\ssc{3-\nu}(1-\alpha_{k})}{\ssc{2}\rho_{k}(1+\ell_{V,\Psi})}\right)}\\
&-(1-\alpha_{k})\textstyle{\left(\frac{\ssc{3-\nu}}{\ssc{2}\rho_{k}(1+\ell_{V,\Psi})}-1\right)}\norm{X_{k+1}-X_{k}}^{2}_{\Psi}.
\end{align*}
\end{lemma}
\begin{proof}
Start by observing 
\begin{align*}
\norm{Z_{k}-p}^{2}_{\Psi}&=\norm{Z_{k}-Y_{k}+Y_{k}-R_{k}+R_{k}-p}^{2}_{\Psi}\\
                         &=\norm{Z_{k}-Y_{k}}^{2}_{\Psi} \us{ \ - \ } \norm{Y_{k}-R_{k}}^{2}_{\Psi}+\norm{R_{k}-p}^{2}_{\Psi}\\
&+2\inner{Z_{k}-R_{k},Y_{k}-p}_{\Psi}.
\end{align*}
Since 
\begin{align} \notag
\norm{Y_{k}-R_{k}}^{2}_{\Psi}&=\norm{\Psi^{-1}(\scrA_{k}-\scrB_{k})}^{2}_{\Psi}\\
&= \norm{V(Y_{k})-V(Z_{k})+W_{k}-U_{k}}^{2}_{\Psi^{-1}} \notag\\
&\leq 2\ell^{2}_{V,\Psi}\norm{Y_{k}-Z_{k}}^{2}_{\Psi}+2\norm{W_{k}-U_{k}}^{2}_{\Psi^{-1}} \notag\\
&=2\ell^{2}_{V,\Psi}\norm{Y_{k}-Z_{k}}^{2}_{\Psi}+2\norm{\us{\ce_{k}}}^{2}_{\Psi^{-1}}.\label{bd-YkRk}
\end{align}
Hence, 
\begin{align*}
    \norm{Z_{k}-p}^{2}_{\Psi}&\us{\overset{\eqref{bd-YkRk}}{\geq}} (1-2\ell^{2}_{V,\Psi})\norm{Z_{k}-Y_{k}}^{2}_{\Psi}-2\norm{\us{\ce_{k}}}^{2}_{\Psi^{-1}}\\
&+\norm{R_{k}-p}^{2}_{\Psi}+2\inner{Z_{k}-R_{k},Y_{k}-p}_{\Psi}.
\end{align*}
\us{Using the definition $Y_{k} \us{ \triangleq \ } \mathrm{J}_{\Psi^{-1}T}(Z_{k}-\Psi^{-1}\scrA_{k})$, we get 
\begin{align*}
    Y_k + \Psi^{-1} T(Y_k) & \ni (Z_k - \Psi^{-1} \scrA_k) \\
    \mbox{ or }
    T(Y_k) & \ni \Psi ( Z_k - Y_k - \Psi^{-1} \scrA_k) 
\end{align*}
Since $p \in \zer(T+V)$,  $(p,{\bf 0}) \in \mbox{gr}(T+V)$, implying that ${\bf 0} - V(p) \in  T(p)$. Consequently, by monotonicity of $T$, we have that 
\begin{align*}
    \inner{\Psi ( Z_k - Y_k - \Psi^{-1} \scrA_k)+V(p), Y_k - p} \geq 0 \\ 
    \mbox{ or } \inner{Z_k-Y_k-\Psi^{-1}(\scrA_k - \scrB_k),Y_k-p}_{\Psi} \\
    \geq \inner{V(Y_k) - V(p),Y_k-p} + \inner{\scrB_k - V(Y_k), Y_k-p}.  
\end{align*}
By definition, $R_k = Y_k+\Psi^{-1}(\scrA_k - \scrB_k)$, $W_k = \scrB_k  - V(Y_k)$, and $V(Y_k) =  \mathbb{E}[\scrB_k \mid \scrG_k]$, we have that 
}
\begin{align*}
\inner{Z_{k}-R_{k},Y_{k}-p}_{\Psi}&\geq \inner{V(Y_{k})-V(p),Y_{k}-p}\\
&+\inner{W_{k},Y_{k}-p}
\end{align*}
Since $V$ is a monotone operator, this implies $\inner{Z_{k}-R_{k},Y_{k}-p}_{\Psi}\geq \inner{W_{k},Y_{k}-p}.$
Whence, 
\begin{align*}
\norm{Z_{k}-p}^{2}_{\Psi}\geq &(1-2\ell^{2}_{V,\Psi})\norm{Y_{k}-Z_{k}}^{2}_{\Psi}+\norm{R_{k}-p}^{2}_{\Psi}\\
&-2\norm{\us{\ce_{k}}}^{2}_{\Psi^{-1}}+2\inner{W_{k},Y_{k}-p}.
\end{align*}
Rearranging, we arrive at 
\begin{equation}\label{eq:Rk}
\begin{array}{l}
    \norm{R_{k}-p}^{2}_{\Psi}\leq \norm{Z_{k}-p}^{2}_{\Psi}+2\norm{\us{\ce_{k}}}^{2}_{\Psi^{-1}}\\
-(1-2\ell^{2}_{V,\Psi})\norm{Y_{k}-Z_{k}}^{2}_{\Psi}+2\inner{W_{k},p-Y_{k}}.
\end{array}
\end{equation}
Next, we use Lemma \ref{lem:ab} to arrive at  
\begin{align*}
&\norm{X_{k+1}-p}^{2}_{\Psi}=\norm{(1-\rho_{k})Z_{k}+\rho_{k}R_{k}-p}^{2}_{\Psi}\\
&=(1-\rho_{k})\norm{Z_{k}-p}^{2}_{\Psi}+\rho_{k}\norm{R_{k}-p}^{2}_{\Psi}\\
&-\rho_{k}(1-\rho_{k})\norm{R_{k}-Z_{k}}^{2}_{\Psi}\\
&=(1-\rho_{k})\norm{Z_{k}-p}^{2}_{\Psi}+\rho_{k}\norm{R_{k}-p}^{2}_{\Psi}-\textstyle{\textstyle{\frac{1-\rho_{k}}{\rho_{k}}}\norm{X_{k+1}-Z_{k}}^{2}_{\Psi}}\\
&\leq \norm{Z_{k}-p}^{2}_{\Psi}-\textstyle{\frac{1-\rho_{k}}{\rho_{k}}}\norm{X_{k+1}-Z_{k}}^{2}_{\Psi}+2\lambda^{2}\rho_{k}\norm{\ce_{k}}^{2}_{\Psi^{-1}}\\
&-\rho_{k}(1-2\ell^{2}_{V,\Psi})\norm{Z_{k}-Y_{k}}^{2}_{\Psi}-2\rho_{k}\inner{W_{k},Y_{k}-p}\\
&=\norm{Z_{k}-p}^{2}_{\Psi}-\textstyle{\textstyle{\frac{1-\rho_{k}}{\rho_{k}}}}\norm{X_{k+1}-Z_{k}}^{2}_{\Psi}-2\rho_{k}\inner{W_{k},Y_{k}-p}\\
&-\rho_{k}(\ssc{(1-\nu)}-2\ell^{2}_{V,\Psi})\norm{Z_{k}-Y_{k}}^{2}_{\Psi}-\textstyle{\textstyle{\ssc{\nu\rho_k}}\norm{Y_{k}-Z_{k}}^{2}_{\Psi}}\\
&+2\rho_{k}\norm{\ce_{k}}^{2}_{\Psi^{-1}}.
\end{align*}
Using \eqref{eq:YZG}, this implies 
\begin{align*}
&\norm{X_{k+1}-p}^{2}_{\Psi}\leq \norm{Z_{k}-p}^{2}_{\Psi}-\textstyle{\frac{1-\rho_{k}}{\rho_{k}}}\norm{X_{k+1}-Z_{k}}^{2}_{\Psi}\\
&-\rho_{k}(\ssc{(1-\nu)}-2\ell^{2}_{V,\Psi})\norm{Z_{k}-Y_{k}}^{2}_{\Psi}-\textstyle{\ssc{\frac{\nu\rho_{k}}{2}}}r^{2}_{\Psi}(Z_{k})\\
&-2\rho_{k}\inner{W_{k},Y_{k}-p}+2\rho_{k}\norm{\ce_{k}}^{2}_{\Psi^{-1}}+\textstyle{\ssc{\nu}\rho_{k}}\norm{U_{k}}^{2}_{\Psi^{-1}}.
\end{align*}
Furthermore,
\begin{align*}
\tfrac{1}{\rho_{k}}\norm{X_{k+1}-Z_{k}}_{\Psi} &=\norm{R_{k}-Z_{k}}_{\Psi}\\
&\leq \norm{\scrB_{k}-\scrA_{k}}_{\Psi^{-1}}+\norm{Y_{k}-Z_{k}}_{\Psi}\\
&\leq(1+ \ell_{V,\Psi})\norm{Y_{k}-Z_{k}}_{\Psi}+\norm{\ce_{k}}_{\Psi^{-1}},
\end{align*}
which implies
\begin{align*}
\tfrac{1}{2\rho^{2}_{k}}\norm{X_{k+1}-Z_{k}}^{2}_{\Psi}\leq (1+\ell_{V,\Psi})^{2}\norm{Y_{k}-Z_{k}}^{2}_{\Psi}+\norm{\ce_{k}}^{2}_{\Psi^{-1}}.
\end{align*}
Multiplying both sides by $\frac{\rho_{k}(\ssc{1-\nu}-2\ell_{V,\Psi})}{1+\ell_{V,\Psi}}$, we obtain 
\begin{align*}
&\textstyle{\frac{\ssc{1-\nu}-2\ell_{V,\Psi}}{2\rho_{k}}(1+\ell_{V,\Psi})}\norm{X_{k+1}-Z_{k}}^{2}_{\Psi}\\
&\leq\rho_{k}(\ssc{1-\nu}-2\ell_{V,\Psi})(1+\ell_{V,\Psi})\norm{Y_{k}-Z_{k}}^{2}_{\Psi}\\
&+\textstyle{\frac{\rho_{k}(\ssc{1-\nu}-2\ell_{V,\Psi})}{1+\ell_{V,\Psi}}}\norm{\ce_{k}}^{2}_{\Psi^{-1}}.
\end{align*}
Rearranging terms, and noting that $(\ssc{1-\nu}-2\ell_{V,\Psi})(1+\ell_{V,\Psi})\leq \ssc{1-\nu}-2\ell^{2}_{V,\Psi}$, the above estimate becomes
\begin{align*}
&-\rho_{k}(\ssc{1-\nu}-2\ell^{2}_{V,\Psi})\norm{Y_{k}-Z_{k}}^{2}_{\Psi}\\
&\leq -\textstyle{\frac{\ssc{1-\nu}-2\ell_{V,\Psi}}{2\rho_{k}}(1+\ell_{V,\Psi})}\norm{X_{k+1}-Z_{k}}^{2}_{\Psi}+\textstyle{\frac{\rho_{k}(\ssc{1-\nu}-2\ell_{V,\Psi}\Psi)}{1+\ell_{V,\Psi}}}\norm{\ce_{k}}^{2}_{\Psi^{-1}}.
\end{align*}
Substituting this bound into the first majorization of the anchor process $\norm{X_{k+1}-p}^{2}_{\Psi}$, we see
\begin{align*}
&\norm{X_{k+1}-p}^{2}_{\Psi}\leq \norm{Z_{k}-p}^{2}_{\Psi}+\textstyle{\ssc{\nu}\rho_{k}}\norm{U_{k}}^{2}_{\Psi^{-1}}-\textstyle{\frac{\ssc{\nu}\rho_{k}}{\ssc{2}}}r^{2}_{\Psi}(Z_{k})\\
&-\left(\textstyle{\frac{1-\rho_{k}}{\rho_{k}}}+\frac{\ssc{1-\nu}-2\ell_{V,\Psi}}{2\rho_{k}(1+\ell_{V,\Psi})}\right)\norm{X_{k+1}-Z_{k}}^{2}_{\Psi}\\
&+\rho_{k}\norm{\ce_{k}}^{2}_{\Psi^{-1}}\textstyle{\left(2+\frac{\ssc{1-\nu}-2\ell_{V,\Psi}}{1+\ell_{V,\Psi}}\right)}-2\rho_{k}\inner{W_{k},Y_{k}-p}.
%&=\norm{Z_{k}-p}^{2}-\textstyle{\frac{\rho_{k}}{4}}g^{2}_{\Psi}(Z_{k})+\frac{\rho_{k}\lambda^{2}_{k}}{2}\norm{U_{k}}^{2}-2\rho_{k}\Psi\inner{W_{k+1}+p^{\ast},Y_{k}-p}\\
%&-\frac{5/2-2\rho_{k}(1+L\Psi)}{2\rho_{k}(1+L\Psi)}\norm{X_{k+1}-Z_{k}}^{2}+\frac{5\rho_{k}\Psi^{2}}{2(1+L\Psi)} \norm{\ce_{k+1}}^{2}+2\rho_{k}\Psi\inner{V(Y_{k})-V(p),p-Y_{k}}.
\end{align*}
Observe that 
\begin{align}\label{eq:Xk1}
\norm{X_{k+1}-Z_{k}}^{2}_{\Psi}&\geq (1-\alpha_{k})\norm{X_{k+1}-X_{k}}^{2}_{\Psi}\nonumber\\
&+(\alpha^{2}_{k}-\alpha_{k})\norm{X_{k}-X_{k-1}}^{2}_{\Psi},
\end{align} 
\begin{align}\label{eq:Z1}
\norm{Z_{k}-p}^{2}_{\Psi}&=(1+\alpha_{k})\norm{X_{k}-p}^{2}_{\Psi}-\alpha_{k}\norm{X_{k-1}-p}^{2}_{\Psi}\nonumber\\
&+\alpha_{k}(1+\alpha_{k})\norm{X_{k}-X_{k-1}}^{2}_{\Psi}.
\end{align}
Choose parameters $\alpha_{k}$ and $\rho_{k}$ such that $\frac{\ssc{3-\nu}-2\rho_{k}(1+\ell_{V,\Psi})}{2\rho_{k}(1+\ell_{V,\Psi})}>0$. Then, using both of these relations in the last estimate for $\norm{X_{k+1}-p}^{2}_{\Psi}$, we arrive at 
\begin{align*}
&\norm{X_{k+1}-p}^{2}_{\Psi}\leq (1+\alpha_{k})\norm{X_{k}-p}^{2}_{\Psi}-\alpha_{k}\norm{X_{k-1}-p}^{2}_{\Psi}\\
&+\alpha_{k}(1+\alpha_{k})\norm{X_{k}-X_{k-1}}^{2}_{\Psi}-2\rho_{k}\inner{W_{k+1},Y_{k}-p}\\
&-\textstyle{\frac{\nu\rho_{k}}{2}}r^{2}_{\Psi}(Z_{k})+\frac{(3-\nu)\rho_{k}}{1+\ell_{V,\Psi}}\norm{\ce_{k}}^{2}_{\Psi^{-1}}+\textstyle{\nu\rho_{k}}\norm{U_{k}}^{2}_{\Psi^{-1}}\\
&-\textstyle{\left(\frac{3-\nu}{2\rho_{k}(1+\ell_{V,\Psi})}-1\right)}[(1-\alpha_{k})\norm{X_{k+1}-X_{k}}^{2}_{\Psi}\\
&+(\alpha^{2}_{k}-\alpha_{k})\norm{X_{k}-X_{k-1}}^{2}_{\Psi}].
\end{align*}
Using the respective definitions of the stochastic increments $\Delta M_{k+1},\Delta N_{k}(p)$ in \eqref{eq:M} and \eqref{eq:N}, we arrive at
\begin{align*}
&\norm{X_{k+1}-p}^{2}_{\Psi}\leq (1+\alpha_{k})\norm{X_{k}-p}^{2}_{\Psi}-\alpha_{k}\norm{X_{k-1}-p}^{2}_{\Psi}\\
&+\Delta M_{k}+\Delta N_{k}(p)-\textstyle{\frac{\nu\rho_{k}}{2}}r^{2}_{\Psi}(Z_{k})\\
&+\alpha_{k}\norm{X_{k}-X_{k-1}}^{2}_{\Psi}\textstyle{\left(2\alpha_{k}+\frac{(3-\nu)(1-\alpha_{k})}{2\rho_{k}(1+\ell_{V,\Psi})}\right)}\\
&-(1-\alpha_{k})\textstyle{\left(\frac{3-\nu}{2\rho_{k}(1+\ell_{V,\Psi})}-1\right)}\norm{X_{k+1}-X_{k}}^{2}_{\Psi}
\end{align*}
\end{proof}
Rearranging the fundamental recursion, we see 
\begin{align*}
& \norm{X_{k+1}-p}^{2}_{\Psi}-\alpha_k\norm{X_{k}-p}^{2}_{\Psi}\\
&+(1-\alpha_k)\textstyle{\left(\frac{3-\nu}{2\rho_k(1+\ell_{V,\Psi})}-1\right)}\norm{X_{k+1}-X_{k}}^{2}_{\Psi}\\
&\leq \norm{X_{k}-p}_{\Psi}^{2}-\alpha_k\norm{X_{k-1}-p}^{2}_{\Psi}+\Delta M_{k}+\Delta N_{k}(p)\\
&+(1-\alpha_k)\textstyle{\left(\frac{3-\nu}{2\rho_k(1+\ell_{V,\Psi})}-1\right)}\norm{X_{k}-X_{k-1}}^{2}_{\Psi}-\tfrac{\nu\rho_{k}}{2}r^2_{\Psi}(Z_k)\\
&+\textstyle{\left(2\alpha_{k}^2+(1-\alpha_{k})\left(1-\tfrac{(3-\nu)(1-\alpha_k)}{2\rho_k(1+\ell_{V,\Psi})}\right)\right)}\norm{X_{k}-X_{k-1}}^{2}_{\Psi}.
\end{align*}
Suppose  $(\alpha_k)_{k}$ is a non-decreasing sequence satisfying  $0<\alpha_k\le\bar{\alpha}<1$ and 
$\rho_k = {\tfrac{(3-\nu)(1-\bar{\alpha})^2}{2(2\alpha_k^2-\alpha_k+1)(1+\ell_{V,\Psi})}}.$ Since $\rho_k\le\tfrac{(3-\nu)(1-\alpha_k)^2}{2(2\alpha_k^2-\alpha_k+1)(1+\ell_{V,\Psi})}$, we claim that 
\begin{align*}
&\norm{X_{k+1}-p}^{2}_{\Psi}-\alpha_{k}\norm{X_{k}-p}^{2}_{\Psi}\\
&+(1-\alpha_k)\left(\tfrac{3-\nu}{2\rho_k(1+\ell_{V,\Psi})}-1\right)\norm{X_{k+1}-X_{k}}^{2}_{\Psi}\\
&\ge \norm{X_{k+1}-p}^{2}_{\Psi}-\alpha_k\norm{X_{k}-p}^{2}_{\Psi}\\
&+(1-\alpha_k)\left(\tfrac{2\alpha_k^2-\alpha_k+1}{(1-\alpha_k)^2}-1\right)\norm{X_{k+1}-X_{k}}^{2}_{\Psi}\geq 0.
\end{align*}
To see this, observe that for any $\alpha>0$,
\begin{align*}
% &\norm{X_{k+1}-p}_{\Psi}^{2}-\alpha\norm{X_k-p}^{2}_{\Psi} \\
%&+ (1-\alpha)\left(\frac{5}{4\rho(1+\ell_{V,\Psi})}-1\right)\norm{X_{k+1}-X_k}^{2}_{\Psi}\\
%&\geq 
&\norm{X_{k+1}-p}_{\Psi}^{2}-\alpha\norm{X_k-p}^{2}_{\Psi} \\
&+(1-\alpha)\textstyle{\left(\frac{2\alpha^2-\alpha+1}{(1-\alpha)^2}-1\right)}\norm{X_{k+1}-X_k}_{\Psi}^{2}\\
&> \norm{X_{k+1}-p}^{2}_{\Psi}-\alpha\norm{X_k-p}^{2}_{\Psi} \\
&+(1-\alpha)\textstyle{\left(\frac{\alpha^2-\alpha+1}{(1-\alpha)^2}-1\right)}\norm{X_{k+1}-X_{k}}^{2}_{\Psi}\\
&= \norm{X_{k+1}-p}^{2}_{\Psi}-\alpha\norm{X_k-p}^{2}_{\Psi}\\
&-\left(\tfrac{\alpha^2-\alpha+1}{1-\alpha}-\tfrac{1-2 \alpha + \alpha^2}{1-\alpha}\right)\norm{X_{k+1}-X_{k}}^{2}_{\Psi} \\
&= (\alpha+(1-\alpha))\norm{X_{k+1}-p}^{2}_{\Psi}-\alpha\norm{X_{k}-p}^{2}_{\Psi} \\
&+\textstyle{\left(\alpha+\frac{\alpha^2}{1-\alpha}\right)}\norm{X_{k+1}-X_{k}}^{2}_{\Psi}\\
&\geq  \alpha\norm{X_{k+1}-p}^{2}_{\Psi}+\alpha\norm{X_{k+1}-X_{k}}^{2}_{\Psi}-\alpha\norm{X_{k}-p}^{2}_{\Psi} \\
&+2\alpha\norm{X_{k+1}-p}_{\Psi}\cdot \norm{X_{k+1}-X_{k}}_{\Psi}\\
&=  \alpha (\norm{X_{k+1}-p}_{\Psi}+\norm{X_{k+1}-X_{k}}_{\Psi})^2-\alpha\norm{X_k-p}^{2}_{\Psi}  \\
& \geq \alpha \norm{X_{k+1}-p+X_{k}-X_{k+1}}^{2}_{\Psi}-\alpha\norm{X_{k}-p}^{2}_{\Psi}= 0
\end{align*}
where the third inequality follows from Young's inequality. Under this specific coupling of the inertial and relaxation parameters, it holds that $2\alpha_{k}^2+(1-\alpha_{k})\left(1-\tfrac{(3-\nu)(1-\alpha_k)}{2\rho_k(1+\ell_{V,\Psi})}\right)\leq 0$. Now, let $H_{k}(p)= \norm{X_{k}-p}_{\Psi}^{2}-\alpha_k\norm{X_{k-1}-p}^{2}_{\Psi}+(1-\alpha_k)\left(\frac{3-\nu}{2\rho_k(1+\ell_{V,\Psi})}-1\right)\norm{X_{k}-X_{k-1}}^{2}_{\Psi}$, and $\delta_{k}= \tfrac{\nu\rho_{k}}{2}r^2_{\Psi}(Z_k)-\left(2\alpha_{k}^2+(1-\alpha_{k})\left(1-\tfrac{(3-\nu)(1-\alpha_k)}{2\rho_k(1+\ell_{V,\Psi})}\right)\right)\norm{X_{k}-X_{k-1}}^{2}_{\Psi}.$ Then, $(1-\alpha_{k+1})\left(\tfrac{3-\nu}{2\rho_{k+1}(1+\ell_{V,\Psi})}-1\right)\norm{X_{k+1}-X_{k}}^{2}_{\Psi} \leq (1-\alpha_{k})\left(\tfrac{3-\nu}{2\rho_{k}(1+\ell_{V,\Psi})}-1\right)\norm{X_{k+1}-X_{k}}^{2}_{\Psi}$. Therefore, for all $k\geq 0$, we conclude 
\begin{align*}
\Ex[H_{k+1}(p)\mid\scrF_{k}]\leq H_{k}(p)-\delta_{k}(p)+\Ex[\Delta M_{k}\mid\scrF_{k}].
\end{align*}
Using Standing Assumption \ref{ass:variance}, we deduce that $\left(\Ex[\Delta M_{k}\mid\scrF_{k}]\right)_{k\in\N}$ is summable, and thus we can apply Lemma \ref{lem:RS} to the above recursion. Hence, we readily deduce the existence of an a.s. finite limiting random variable $H_{\infty}(p)$ such that $\Pr\left(\lim_{k\to\infty}H_{k}(p)=H_{\infty}(p)\right)=1$ and $\sum_{k\in\NN}\delta_{k}(p))_{k\in\N}<\infty$. Therefore, there exists a measurable set $\Omega_{0}\in\scrF$ with $\Pr(\Omega_{0})=1$ such that for each $\omega\in\Omega_{0}$ it holds true that 
\begin{align*}
&\lim_{k\to\infty}\norm{X_{k}(\omega)-X_{k-1}(\omega)}_{\Psi}=0 \text{ and }\\
&\lim_{k\to\infty}\textstyle{\frac{\nu\rho_{k}}{2}}r^{2}_{\Psi}(Z_{k}(\omega))=0.
\end{align*}
Assuming that $\liminf_{k\to\infty}\rho_{k}>0$, we conclude that $\lim_{k\to\infty}r^{2}_{\Psi}(Z_{k})=0$ $\Pr$-a.s. Therefore, we conclude that $(X_{k})_{k\geq 1}$ converges a.s. to a limiting random variable with values in $\zer(V+T)$.

%----------------------------------------------------------------------
%%% CONCLUSIONS
%----------------------------------------------------------------------
\section{Conclusion}
\label{sec:conclusion}

In the context of \us{shared constraint variants of} stochastic generalized Nash equilibrium problems, the convergence of the forward-backward-forward algorithm can be boosted via an integrated acceleration-relaxation procedure. In \us{the} presence of stochastic uncertainty, convergence can be proved assuming only monotonicity and Lipschitz continuity of the expected-valued operator. Specifically, our main result is the \us{claim of} almost sure global convergence of the trajectory of actions to the set of variational equilibria. In future research, we aim to investigate the question how \us{one may} relax monotonicity and Lipschitz continuity assumptions even further, \us{derive rate statements}, and \us{examine} how partial information \us{may} be introduced into the algorithm.
%*************************************************************
%*****    APPENDICES
%*************************************************************
%\section{appendix}
%\subsection{Analysis of algorithm SFBF}
%\label{sec:SFBF}
%\input{AnalysisSFBF}

%\subsection{Analysis of algorithm RISFBF}
%\label{sec:RISFBF}
%\input{AnalysisRISFBF}

\bibliographystyle{IEEEtran}
\bibliography{IEEEabrv,mybib,BixBiblio}

\begin{thebibliography}{10}
\providecommand{\url}[1]{#1}
\csname url@samestyle\endcsname
\providecommand{\newblock}{\relax}
\providecommand{\bibinfo}[2]{#2}
\providecommand{\BIBentrySTDinterwordspacing}{\spaceskip=0pt\relax}
\providecommand{\BIBentryALTinterwordstretchfactor}{4}
\providecommand{\BIBentryALTinterwordspacing}{\spaceskip=\fontdimen2\font plus
\BIBentryALTinterwordstretchfactor\fontdimen3\font minus
  \fontdimen4\font\relax}
\providecommand{\BIBforeignlanguage}[2]{{%
\expandafter\ifx\csname l@#1\endcsname\relax
\typeout{** WARNING: IEEEtran.bst: No hyphenation pattern has been}%
\typeout{** loaded for the language `#1'. Using the pattern for}%
\typeout{** the default language instead.}%
\else
\language=\csname l@#1\endcsname
\fi
#2}}
\providecommand{\BIBdecl}{\relax}
\BIBdecl

\bibitem{kulkarni2012}
A.~A. Kulkarni and U.~V. Shanbhag, ``On the variational equilibrium as a
  refinement of the generalized {Nash} equilibrium,'' \emph{Automatica},
  vol.~48, no.~1, pp. 45--55, 2012.

\bibitem{ravat2011}
U.~Ravat and U.~V. Shanbhag, ``On the characterization of solution sets of
  smooth and nonsmooth convex stochastic {Nash} games,'' \emph{SIAM Journal on
  Optimization}, vol.~21, no.~3, pp. 1168--1199, 2011.

\bibitem{MerStaIFAC19}
M.~Staudigl and P.~Mertikopoulos, ``Convergent noisy forward-backward-forward
  algorithms in non-monotone variational inequalities,''
  \emph{IFAC-PapersOnLine}, vol.~52, no.~3, pp. 120--125, 2019.

\bibitem{yu2017}
C.-K. Yu, M.~Van Der~Schaar, and A.~H. Sayed, ``Distributed learning for
  stochastic generalized {Nash} equilibrium problems,'' \emph{IEEE Transactions
  on Signal Processing}, vol.~65, no.~15, pp. 3893--3908, 2017.

\bibitem{YeHu16}
M.~Ye and G.~Hu, ``Game design and analysis for price-based demand response: An
  aggregate game approach,'' \emph{IEEE transactions on cybernetics}, vol.~47,
  no.~3, pp. 720--730, 2016.

\bibitem{kannan13addressing}
\BIBentryALTinterwordspacing
A.~Kannan, U.~V. Shanbhag, and H.~M. Kim, ``Addressing supply-side risk in
  uncertain power markets: stochastic nash models, scalable algorithms and
  error analysis,'' \emph{Optim. Methods Softw.}, vol.~28, no.~5, pp.
  1095--1138, 2013. [Online]. Available:
  \url{https://doi.org/10.1080/10556788.2012.676756}
\BIBentrySTDinterwordspacing

\bibitem{watling2006}
D.~Watling, ``User equilibrium traffic network assignment with stochastic
  travel times and late arrival penalty,'' \emph{European Journal of
  Operational Research}, vol. 175, no.~3, pp. 1539--1556, 2006.

\bibitem{henrion2007}
R.~Henrion and W.~R{\"o}misch, ``On m-stationary points for a stochastic
  equilibrium problem under equilibrium constraints in electricity spot market
  modeling,'' \emph{Applications of Mathematics}, vol.~52, no.~6, pp. 473--494,
  2007.

\bibitem{abada2013}
I.~Abada, S.~Gabriel, V.~Briat, and O.~Massol, ``A generalized
  {Nash}--{Cournot} model for the {Northwestern} {European} natural gas markets
  with a fuel substitution demand function: The gammes model,'' \emph{Networks
  and Spatial Economics}, vol.~13, no.~1, pp. 1--42, 2013.

\bibitem{yi2019}
P.~Yi and L.~Pavel, ``An operator splitting approach for distributed
  generalized {Nash} equilibria computation,'' \emph{Automatica}, vol. 102, pp.
  111--121, 2019.

\bibitem{alvarez2001}
F.~Alvarez and H.~Attouch, ``An inertial proximal method for maximal monotone
  operators via discretization of a nonlinear oscillator with damping,''
  \emph{Set-Valued Analysis}, vol.~9, no.~1, pp. 3--11, 2001.

\bibitem{attouch2019}
H.~Attouch and A.~Cabot, ``Convergence of a relaxed inertial forward--backward
  algorithm for structured monotone inclusions,'' \emph{Applied Mathematics \&
  Optimization}, vol.~80, no.~3, pp. 547--598, 2019.

\bibitem{attouch2020}
------, ``Convergence of a relaxed inertial proximal algorithm for maximally
  monotone operators,'' \emph{Mathematical Programming}, vol. 184, no.~1, pp.
  243--287, 2020.

\bibitem{lorenz2015}
D.~A. Lorenz and T.~Pock, ``An inertial forward-backward algorithm for monotone
  inclusions,'' \emph{Journal of Mathematical Imaging and Vision}, vol.~51,
  no.~2, pp. 311--325, 2015.

\bibitem{Nes83}
Y.~Nesterov, ``A method of solving a convex programming problem with
  convergence rate $o(1/k^{2})$.'' \emph{Soviet Mathematics Doklady}, vol.~27,
  no.~2, pp. 372--376, 1983.

\bibitem{Bia16}
P.~Bianchi, ``Ergodic convergence of a stochastic proximal point algorithm,''
  \emph{SIAM Journal on Optimization}, vol.~26, no.~4, pp. 2235--2260, 2016.

\bibitem{rosasco2016in}
L.~Rosasco, S.~Villa, and B.~C. V{\~u}, ``A stochastic inertial
  forward--backward splitting algorithm for multivariate monotone inclusions,''
  \emph{Optimization}, vol.~65, no.~6, pp. 1293--1314, 2016.

\bibitem{rosasco2016}
------, ``Stochastic forward--backward splitting for monotone inclusions,''
  \emph{Journal of Optimization Theory and Applications}, vol. 169, no.~2, pp.
  388--406, 2016.

\bibitem{CuiSha20}
S.~Cui and U.~V. Shanbhag, ``Variance-reduced proximal and splitting schemes
  for monotone stochastic generalized equations,'' \emph{arXiv preprint
  arXiv:2008.11348}, 2020.

\bibitem{KanSha19}
A.~Kannan and U.~V. Shanbhag, ``Optimal stochastic extragradient schemes for
  pseudomonotone stochastic variational inequality problems and their
  variants,'' \emph{Computational Optimization and Applications}, vol.~74,
  no.~3, pp. 779--820, 2019.

\bibitem{SFBF}
R.~I. Bot, P.~Mertikopoulos, M.~Staudigl, and P.~T. Vuong, ``Mini-batch
  forward-backward-forward methods for solving stochastic variational
  inequalities,'' \emph{Forthcoming: Stochastic Systems}, 2021.

\bibitem{Tse00}
\BIBentryALTinterwordspacing
P.~Tseng, ``A modified forward-backward splitting method for maximal monotone
  mappings,'' \emph{SIAM Journal on Control and Optimization}, vol.~38, no.~2,
  pp. 431--446, 2018/09/13 2000. [Online]. Available:
  \url{https://doi.org/10.1137/S0363012998338806}
\BIBentrySTDinterwordspacing

\bibitem{BauCom16}
H.~H. Bauschke and P.~L. Combettes, \emph{Convex Analysis and Monotone Operator
  Theory in Hilbert Spaces}.\hskip 1em plus 0.5em minus 0.4em\relax Springer -
  CMS Books in Mathematics, 2016.

\bibitem{franci2019fbf}
B.~Franci, S.~Grammatico, and M.~Staudigl, ``Distributed forward-backward
  (half) forward algorithms for generalized {Nash} equilibrium seeking,'' in
  \emph{European Control Conference (ECC), St. Petersburg, Russia,}, 2020.

\bibitem{KulSha12}
\BIBentryALTinterwordspacing
A.~A. Kulkarni and U.~V. Shanbhag, ``On the variational equilibrium as a
  refinement of the generalized nash equilibrium,'' \emph{Automatica}, vol.~48,
  no.~1, pp. 45--55, 2012. [Online]. Available:
  \url{http://www.sciencedirect.com/science/article/pii/S0005109811004821}
\BIBentrySTDinterwordspacing

\bibitem{franci2020fbtac}
B.~Franci and S.~Grammatico, ``A distributed forward-backward algorithm for
  stochastic generalized nash equilibrium seeking,'' \emph{IEEE Transactions on
  Automatic Control}, 2020.

\bibitem{IusJofOliTho17}
A.~Iusem, A.~Jofr{\'e}, R.~I. Oliveira, and P.~Thompson, ``Extragradient method
  with variance reduction for stochastic variational inequalities,'' \emph{SIAM
  Journal on Optimization}, vol.~27, no.~2, pp. 686--724 %@ 1052--6234, 2017.

\bibitem{facchineikanzow2007}
F.~Facchinei and C.~Kanzow, ``Generalized nash equilibrium problems,''
  \emph{4or}, vol.~5, no.~3, pp. 173--210, 2007.

\bibitem{facchinei2007vi}
F.~Facchinei, A.~Fischer, and V.~Piccialli, ``On generalized {Nash} games and
  variational inequalities,'' \emph{Operations Research Letters}, vol.~35,
  no.~2, pp. 159--164, 2007.

\bibitem{FacPan03}
F.~Facchinei and J.-s. Pang, \emph{Finite-Dimensional Variational Inequalities
  and Complementarity Problems - Volume I and Volume II}.\hskip 1em plus 0.5em
  minus 0.4em\relax Springer Series in Operations Research, 2003.

\bibitem{auslender2000}
A.~Auslender and M.~Teboulle, ``Lagrangian duality and related multiplier
  methods for variational inequality problems,'' \emph{SIAM Journal on
  Optimization}, vol.~10, no.~4, pp. 1097--1115, 2000.

\bibitem{godsil2013}
C.~Godsil and G.~F. Royle, \emph{Algebraic graph theory}.\hskip 1em plus 0.5em
  minus 0.4em\relax Springer Science \& Business Media, 2013, vol. 207.

\bibitem{Byrd2012}
\BIBentryALTinterwordspacing
R.~H. Byrd, G.~M. Chin, J.~Nocedal, and Y.~Wu, ``Sample size selection in
  optimization methods for machine learning,'' \emph{Mathematical Programming},
  vol. 134, no.~1, pp. 127--155, 2012. [Online]. Available:
  \url{https://doi.org/10.1007/s10107-012-0572-5}
\BIBentrySTDinterwordspacing

\bibitem{lei2018distributed}
J.~Lei and U.~V. Shanbhag, ``Distributed variable sample-size gradient-response
  and best-response schemes for stochastic nash equilibrium problems over
  graphs,'' \emph{arXiv preprint arXiv:1811.11246}, 2018.

\bibitem{bot2020}
R.~Bot, P.~Mertikopoulos, M.~Staudigl, and P.~Vuong, ``Mini-batch
  forward-backward-forward methods for solving stochastic variational
  inequalities,'' \emph{Stochastic Systems}, 2020.

\bibitem{kannan2012}
A.~Kannan and U.~V. Shanbhag, ``Distributed computation of equilibria in
  monotone nash games via iterative regularization techniques,'' \emph{SIAM
  Journal on Optimization}, vol.~22, no.~4, pp. 1177--1205, 2012.

\bibitem{Pol87}
B.~T. Polyak, \emph{Introduction to Optimization}.\hskip 1em plus 0.5em minus
  0.4em\relax Optimization Software, 1987.

\end{thebibliography}

\end{document}